\documentclass[11pt,reqno]{amsart}

\title[Parametric reflection maps: an algebraic approach]
{Parametric reflection maps: an algebraic approach}
\author[Anastasia Doikou]{Anastasia Doikou}

\author[Marzia Mazzotta]{Marzia Mazzotta}

\author[Paola Stefanelli]{Paola Stefanelli}

\address[Anastasia Doikou]{Department of Mathematics, Heriot-Watt University,
Edinburgh EH14 4AS $\&$ Maxwell Institute for Mathematical Sciences, Edinburgh EH8 9BT, UK\\
ORCID: 0000-0001-6869-9389}
\email{a.doikou@hw.ac.uk}

\address[Marzia Mazzotta $\&$ Paola Stefanelli]{Dipartimento di Matematica e Fisica “Ennio De Giorgi”,  Università 
del Salento, Via Provinciale Lecce-Arnesano, 73100 Lecce, Italy, ORCID (M. Mazzotta): 0000-0001-6179-9862, ORCID (P. Stefanelli): 0000-0003-3899-3151}
\email{marzia.mazzotta@unisalento.it,\ paola.stefanelli@unisalento.it}

  \usepackage{latexsym} 
  \usepackage[all]{xy}
  \usepackage{amsfonts} 
  \usepackage{amsthm} 
  \usepackage{amsmath} 
  \usepackage{amssymb}
  \usepackage{pifont}  
  \usepackage{enumerate}
  \usepackage{dcolumn}
  \usepackage{comment}
  \usepackage{hyperref}  
  \usepackage{tikz-cd}
\usepackage[english]{babel}
  \usepackage{amsfonts} 
\usepackage[utf8]{inputenc}
\usepackage{graphicx}
\usepackage{subcaption}
\usepackage[colorinlistoftodos]{todonotes}
\usepackage{indentfirst}
\usepackage{pifont}  
  \usepackage{enumerate}
  \usepackage{dcolumn}
  \usepackage{comment}
\usepackage[capitalise]{cleveref}

 \usepackage{tikz}
\usetikzlibrary{arrows}

 \newcolumntype{2}{D{.}{}{2.0}}
  \xyoption{2cell}

\newcommand{\hiddenpower}[2] { \ifnum \numexpr#2=1 #1 \else #1^#2 \fi }
\numberwithin{equation}{section}

\def\be{\begin{equation}}
\def\ee{\end{equation}}
\def\ba{\begin{eqnarray}}
\def\ea{\end{eqnarray}}

\newcommand{\cal}{\mathcal}

\usepackage{xparse}
\ExplSyntaxOn
\newcounter{diff_order}
\newcounter{diff_power}

\newcommand{\rawdiff}[3]
{
	\setcounter{diff_order}{0}
	\clist_map_inline:nn{#3}{\stepcounter{diff_order}}
	
	\frac{\hiddenpower{#1}{\thediff_order} #2}
	{
		\def\old_var{DefaultValue}
		\setcounter{diff_power}{0}
		
		\clist_map_inline:nn{#3}
		{
			\def\new_var{##1}
			\ifnum \thediff_power=0
				\stepcounter{diff_power}
			\else
				\tl_if_eq:NNTF \new_var \old_var
				{\stepcounter{diff_power}}
				{
					#1 \hiddenpower{\old_var}{\thediff_power}
					\setcounter{diff_power}{1}
				}
			\fi

			\def\old_var{##1}
		}
		
		#1 \hiddenpower{\old_var}{\thediff_power}
	}
}
\def\Label#1{\label{#1}\ifmmode\llap{[#1] }\else 
  \marginpar{\smash{\hbox{\tiny [#1]}}}\fi} 
  \def\Label{\label} 

\ExplSyntaxOff

\newlength{\bibitemsep}
\newlength{\bibparskip}
\let\oldthebibliography\thebibliography
\renewcommand\thebibliography[1]{%
  \oldthebibliography{#1}%
  \setlength{\parskip}{\bibitemsep}%
  \setlength{\itemsep}{\bibparskip}%
}

\newtheorem{thm}{Theorem}[section]
\newtheorem{lemma}[thm]{Lemma}
\newtheorem{cor}[thm]{Corollary}
\newtheorem{pro}[thm]{Proposition}
\newtheorem{defn}[thm]{Definition}

\newtheorem{rem}[thm]{Remark}
\newtheorem{exa}[thm]{Example}

\newcommand{\Sym}{\operatorname{Sym}}

\newcommand{\id}{\operatorname{id}}

\newenvironment{wideregion}[1][0mm]
    {\vspace{#1}\list{}{\leftmargin-2cm\rightmargin\leftmargin}\item\relax}
    {\endlist}

\newenvironment{widegather }{\wideregion[-9mm]\gather}{\endgather\endwideregion}

\begin{document}
\vskip 0.8in

\hfill
 \begin{abstract} 
We study solutions of the parametric set-theoretic reflection equation from an algebraic perspective by employing recently derived generalizations of the familiar shelves and racks, called parametric ($p$)-shelves and racks.  Generic invertible solutions of the set-theoretic reflection equation are also obtained by a suitable parametric twist. The twist leads to considerably simplified constraints compared to the ones obtained from general set-theoretic reflections. In this context, novel algebraic structures of (skew) $p$-braces that generalize the known (skew) braces and are suitable for the parametric Yang-Baxter equation are introduced.  The $p$-rack Yang-Baxter and reflection operators as well as the associated algebraic structures are defined 
setting up the frame for formulating the $p$-rack reflection algebra.
\end{abstract}
\maketitle





\section*{Introduction} 
\noindent  
The main goal of the present investigation is the study of the parametric set-theoretic reflection equation from a purely algebraic point of view. This is the first time such a study has been undertaken and is highly motivated by recent results on the parametric set-theoretic Yang-Baxter equation in \cite{Doikoup}, but also by recent investigations on the parameter-free reflection equation associated with racks \cite{AlMaSt24x}.

The reflection equation was first introduced and studied in \cite{Cherednik, Sklyanin} in the context of quantum integrable systems offering a systematic method of classifying integrable boundary conditions in one-dimensional integrable lattice systems and two-dimensional integrable quantum field theories. 
The boundary effects, controlled by the reflection equation, shed new light on the bulk theories themselves, and also paved the way to new mathematical concepts and physical applications. The study of the reflection equation goes hand in hand with the study of the Yang-Baxter equation \cite{Baxter, Yang} given that solutions of the Yang-Baxter equation are the basic building blocks of the reflection equation itself. From a physical point of view, especially in the context of quantum integrable systems, solutions of the Yang-Baxter equation characterize the scattering among particle excitations of the quantum integrable systems, whereas the Yang-Baxter equation describes the factorization of many particle scattering. Analogously, solutions of the reflection equation describe the reflection of particle-like excitations to the boundary of the integrable system \cite{Sklyanin, Cherednik}. From a mathematical aspect, the Yang-Baxter equation leads to certain quadratic relations that define various quantum (Hopf) algebras depending on the choice of the solution of the Yang-Baxter equation \cite{FRT, Drinfeld, Jimbo1, Jimbo2}, while the reflection equation provides the defining relations of the so-called reflection algebras, which are typically subalgebras of quantum algebras \cite{Sklyanin}.

In the present study, we focus on solutions of the parametric, set-theoretic Yang-Baxter and reflection equations. The set-theoretic Yang-Baxter equation first introduced by Drinfel'd \cite{Dr92}, whereas a plethora of studies followed on the investigation of solutions of the equation and associated algebraic structures, see for instance \cite{Bachi,CeJeOk14, CoJeKuVaVe23, DoiRyb22, DoRy23}, \cite{GatMaj}-\cite{Gat3}, \cite{EtScSo99, Crystal2, GuaVen, Pili, JesKub, Jesp2, Lebed, LebVen, MaRySt24, Papa, Papa2}, \cite{Ru05}-\cite{Ru19}, \cite{Sol, Veselov}. The parametric set-theoretical reflection equation together with the first examples of solutions appeared for the first time in \cite{Cau2}, while a more systematic study and a classification inspired by maps appearing in integrable discrete systems \cite{Adler, Papa, Veselov} are presented in \cite{CauCra}. In \cite{SmVeWe} methods from the theory of braces were used to produce families of new solutions to the non-parametric reflection equation, and in \cite{Decom} skew braces were used to produce reflections. In \cite{DoiSmo} the reflection algebra associated with involutive, set-theoretic solutions is studied and physical Hamiltonians that describe quantum spin chain systems with integrable boundary conditions coming from set-theoretic solutions are also derived. In \cite{LebVen} the set-theoretic reflection equation is employed for the derivation of solutions of the set-theoretic Yang-Baxter equation. More recently, non-involutive, parameter-free reflection maps associated with racks and quandles were investigated in \cite{AlMaSt24x}. 

Up to date, the parameter-free Yang-Baxter and reflection equations have been investigated primarily from an algebraic perspective, whereas the parametric equations and their solutions have been examined mostly in the context of classical integrable systems in a geometric frame. However, in \cite{Doikoup} an entirely algebraic analysis for the parametric, set-theoretic Yang-Baxter equation was carried out, and purely algebraic solutions were produced. Indeed, utilizing the main definitions and results of \cite{Doikoup} we will be able to proceed with our analysis on the parametric reflection equation. Specifically, we give below a brief account of what is achieved in each of the subsequent sections and what the main results are:
\begin{enumerate}

\item In Section 1 we recall basic definitions that will be used extensively in our analysis. Specifically, in Subsection 1.1  we introduce the parametric set-theoretic reflection equation for a given solution of the parametric set-theoretic Yang-Baxter equation. In Subsection 1.2, we recall the definitions of parametric shelves and racks ($p$-shelves, $p$-racks) introduced in \cite{Doikoup} as objects that satisfy a parametric self-distributivity condition generalizing the definitions of the familiar shelves and racks.  Two new Theorems \ref{prop-rack-p-r} and \ref{prop-inv} on the construction of $p$-racks from the usual racks are also formulated.

\item In Section 2 we proceed with our main analysis of reflection maps obtained from parametric shelves and racks. We first derive the basic conditions emerging from the parametric reflection equation for any reflection map associated with $p$-shelves (Proposition \ref{shelvesf}) and $p$-racks (Corollary \ref{rackf}). We then present an algorithm to obtain reflections associated with $p$-shelves and $p$-racks from the familiar shelves, racks. This is presented in the main Theorem \ref{refl-prop-rack-p-r}. Utilizing the obtained $p$-shelves, $p$-racks we present various examples of relevant reflection maps.

\item Section 3 is devoted to the study of general solutions of the parametric set-theoretic Yang-Baxter and reflection equation obtained via a Drinfel'd twist. In Subsection 3.1, we derive the conditions for reflection maps associated with general solutions of the Yang-Baxter equation, and we then recall that all solutions are obtained from $p$-shelf, $p$-rack solution by means of an admissible Drinfel'd twist. Using the admissible twist we are then able to derive generic parametric reflection maps from rack reflections (Proposition \ref{reflectionA}). Interestingly, the twist leads to considerably simplified constraints compared to the ones obtained from general set-theoretic reflections, facilitating the derivation of general solutions of the parametric set-theoretic reflection equation. In Subsection 3.2, we exploit the results of the previous subsection and present algorithms that systematically lead to general solutions of the parametric Yang-Baxter equation for both the reversible and the general cases. A brief analysis on the algebraic structures emerging in solutions known from classical integrability and the re-factorization technique \cite{Adler, Papa2, Veselov} is also presented. In the case of general bijective solutions we provide several distinct examples of solutions coming from different types of $p$-racks.

\item In Section 4, we introduce the parametric version of braces and skew braces, which we call $p$-braces and skew $p$-braces (see \cref{p-brace} and \cref{skew p-brace}). These are in general non-associative structures and are derived in terms of a reversible $p$-affine and $p$-affine structure on a fixed group (see \cite{Rump1, Rump2, St23} for the non-parametric case).
Such maps allow for construction of $p$-racks, which have as underlying structure the same group.
In addition, we provide admissible Drinfel'd twists for these $p$-racks and then parametric solutions related to them
(see \cref{thm_affine} and \cref{thm_skew_p}).

\item  In the last section, we study the underlying Yang-Baxter and reflection algebras, and their coproduct structures in the set-theoretic frame.
Specifically, we focus on the Yang-Baxter and reflection algebraic structures associated with $p$-rack solutions of the Yang-Baxter equation.
Bearing in mind the definition of braided groups and braidings in \cite{chin}, \cite{GatMaj} 
and their deformations \cite{DoiRyb22, DoRy23, DoRySt} we further generalize the definitions to introduce 
the parametric Yang-Baxter and reflection structures. The
parametric Yang-Baxter operator and associated algebraic structures were studied in \cite{Doikoup}.
More precisely, we define the $p$-rack reflection operator and we show that any such operator satisfies automatically the reflection equation. We also show certain properties of such operators that follow naturally from the definition (see in particular Proposition \ref{le:p-def1}).

    \end{enumerate}

\section{Preliminaries}

\subsection{The parametric reflection equation}
Before we introduce the parametric set-theoretic reflection equation, we recall the parametric set-theoretic Yang-Baxter equation.
Let $X, \ Y \subseteq X$ be non-empty sets,  $z_{i,j} \in Y,$ $i,j \in \mathbb {\mathbb Z}^+$ ($z_{i,j,k,\ldots}$ is shorthand for $z_{i}, z_j, z_k,\ldots$) and consider a map $R^{z_{ij}}:X\times X\rightarrow X\times X$ such that, for all $a, b\in X,$
 $~R^{z_{ij}}(b, a)= \big (\sigma^{z_{ij}} _{a}(b), \tau^{z_{ij}}_{b}(a)\big ).$
The notation $z_{i j}$ denotes dependence 
on $(z_i,  z_j).$  We say that $(X, R^{z_{ij}})$ is a \emph{solution of the parametric, set-theoretic Yang-Baxter equation (or simply a solution)} if 
\begin{equation}
R_{12}^{z_{12}} \ R^{z_{13}}_{13}\ R_{23}^{z_{23}} =  R_{23}^{z_{23}}\ R^{z_{13}}_{13}\  R_{12}^{z_{12}}, \label{YBE}
\end{equation}
where, for all $a,b,c \in X$, $R_{12}^{z_{ij}}(c,b,a) = \left(\sigma^{z_{ij}}_b(c) ,\tau^{z_{ij}}_c(b), a\right),$ 
$R_{13}^{z_{ij}}(c,b,a) = \left(\sigma^{z_{ij}}_a(c) ,b, \tau^{z_{ij}}_c(a)\right)$ and $R_{23}^{z_{ij}}(c,b,a) = \left(c, \sigma^{z_{ij}}_a(b) ,\tau^{z_{ij}}_b(a)\right).$ 
We say that $R^{z_{ij}}$ is \emph{left non-degenerate} if, for all $z_{i,j}\in Y,$ $\sigma^{z_{ij}}_{a}$ is a bijective function and  \emph{non-degenerate} if both $\sigma^{z_{ij}}_{a},\ \tau^{z_{ij}}_{b}$ are bijective functions. Also, the solution $(X, R^{z_{ij}})$  is called \emph{reversible} \cite{Adler, Papa, Papa2} if $R^{z_{21}}_{21} R^{z_{12}}_{12} = \id_{X \times X}$. All solutions from the point of view of discrete integrable systems \cite{Adler, Papa, Papa2, Veselov} or the re-factorization approach are reversible and it was shown in \cite{Doikoup} that they are reduced to the identity map.  

We may now introduce the parametric set-theoretic reflection equation. In this case, in addition to the $R$-maps we also need a map $K^{z}: X \to X,$ such that $a \mapsto \kappa^{z}(a)$. Let also $\mu: Y \to Y$ be an involutive map, such that $z \mapsto \mu(z)$ for all $z \in Y,$ and we introduce the shorthand notation $z_{\bar j} = \mu(z_j),$ for all $z_j \in Y.$
Then the  map $K^z,$ with $z\in Y,$ is called a \emph{parametric reflection} (or simply reflection in this manuscript) 
if it is a solution of the parametric reflection equation:
\begin{equation}\label{reflec_eq}
R_{12}^{z_{12}}\ K_1^{z_1}\ R_{21}^{z_{2\bar 1}}\ K_2^{z_2} = K_2^{z_2}\ R_{12}^{z_{1\bar2}}\ K_1^{z_1}\  R_{21}^{z_{\bar 2 \bar 1}},
\end{equation}
for all $z_{1,2} \in Y$, where for all $a, b \in X$, $z\in Y,$ $K_1^{z}(a,b) = (\kappa^z(a), b)$ and $K_2^z(a,b) = (a, \kappa^{z}(b)).$

It is useful to recall the main proposition about generic-type (or ``dynamical'') 
solutions of the reflection equation (representation of the reflection algebra) 
introduced by Sklyanin \cite{Sklyanin}. 
\begin{pro} (Sklyanin \cite{Sklyanin}) Let $Y \subseteq X$ and for all $z_{i} \in Y,$ $i\in {\mathbb Z}^+$, let $R^{z_{12}}: X\times X \to X \times X$ be a solution of the Yang-Baxter equation and $K^{z_1}: X \to X$ be a solution of the reflection equation. Then, 
$${\mathbb K}_{1,3 \ldots n}^{z_{1,3\ldots n}} : =R^{z_{1 n}}_{1n}\ldots  R^{z_{14}}_{14}R^{z_{13}}_{13} K_1^{z_1} R^{z_{3\bar 1}}_{31} R^{z_{4\bar 1}}_{41}\ldots R^{z_{n\bar 1}}_{n1}$$ 
also satisfies the reflection equation, i.e.
\begin{equation}
R_{12}^{z_{12}}\ {\mathbb K}_1^{z_1}\ R_{21}^{z_{2\bar 1}}\ {\mathbb K}_2^{z_{2}} = {\mathbb K}_2^{z_{2}}\ R_{12}^{z_{1\bar2}}\ {\mathbb K}_1^{z_{1}}\  R_{21}^{z_{\bar 2 \bar 1}}.
\end{equation}
(Notation: usually for simplicity we omit the indices $3, \ldots, n$ in ${\mathbb K}_{1,3\ldots n}^{z_{1,3\ldots n}}$ and simply write ${\mathbb K}_1^{z_{1}}$).
\end{pro}
\begin{proof}
    The proof is based on the repeated use of the Yang-Baxter equation and the reflection equation (see \cite{Sklyanin} for the detailed proof). 
\end{proof}

\subsection{\texorpdfstring{$p$-}{}shelves and \texorpdfstring{$p$-}{}racks} We recall the definitions of $p$-shelves and $p$-racks first introduced in \cite{Doikoup}.
\begin{defn} \label{pp} Let $X, \ Y \subseteq X$ be  non-empty sets and define for all $z_{i,j} \in Y,$ the binary operation $\triangleright_{z_{ij}}: X \times X \to X,$ $(a,b) \mapsto a \triangleright_{z_{ij}} b.$ The pair $\left(X,\,\triangleright_{z_{ij}} \right)$ 
is said to be a \emph{left parametric $p$-shelf} if\, $\triangleright_{z_{ij}}$\, satisfies the generalized left $p$-self-distributivity:
    \begin{equation}\label{shelf}
a\triangleright_{z_{ik}}\left(b\triangleright_{z_{jk}} c\right)
        = \left(a\triangleright_{z_{ij}} b\right)\triangleright_{z_{jk}}\left(a\triangleright_{z_{ik}} c\right) 
    \end{equation}
    for all $a,b,c\in X,$ $z_{i,j,k} \in Y.$ Moreover, a left $p$-shelf $\left(X,\,\triangleright_{z_{ij}} \right)$ is called 
     a \emph{left $p$-rack} if the maps $L^{z_{ij}}_a:X\to X$ defined by $L^{z_{ij}}_a\left(b\right):= a\triangleright_{z_{ij}} b$, for all $a,b, \in X,$ $z_{i,j} \in Y,$ are bijective.
\end{defn}
From now on, whenever we say $p$-shelf or $p$-rack we mean left $p$-shelf or left $p$-rack.

\begin{rem} In the absence of parameters $z_i$, $p$-shelves reduce to the usual shelves (self-distributive structures) $(X, \triangleright),$ and $p$-racks reduce to the familiar racks (see e.g. \cite{FeRo92} and \cite{Shelf-history2}).
\end{rem}

\begin{pro} \label{prose}
Let $X, \ Y \subseteq X$ be non-empty sets and define for $z_{i, j} \in Y$ 
the binary operation $\triangleright_{z_{ij}}:  X \times X \to X,$ $(a,b) \mapsto a \triangleright_{z_{ij}} b.$ Then the map $R^{z_{ij}}:X \times X \to X \times X$, such that for all $a,b \in X, $ $z_{i,j} \in Y,$ $R^{z_{ij}}(b,a) = (b, b \triangleright_{z_{ij}} a)$, is a left non-degenerate solution of the parametric Yang-Baxter equation if and only if $(X, \triangleright_{z_{ij}})$ is a $p$-shelf. In addition, $R^{z_{ij}}$ is invertible and non-degenerate if and only if $(X, \triangleright_{z_{ij}})$ is a $p$-rack. We call such a map $R^{z_{ij}}$ \emph{$p$-shelf or $p$-rack solution}.
\end{pro}
\begin{proof} The proof is contained in \cite{Doikoup}.
\end{proof}

\begin{rem}
If $\left(X,  R^{z_{ij}}\right)$ is a left non-degenerate solution, then the structure $\left(X,  \triangleright_{z_{ij}}\right)$ is a $p$-shelf where $\triangleright_{z_{ij}}$ is defined by $a\triangleright_{z_{ij}} b:= \sigma^{z_{ji}}_a\big(\tau^{z_{ij}}_{\big(\sigma^{z_{ij}}_{b}\big)^{-1}\left(a\right)}\left(b\right)\big)$, for all $a,b\in X$.
\end{rem}


In the following, we provide a method for constructing $p$-shelves  and $p$-racks starting from a fixed shelf or rack. For preliminaries on left shelves, racks, and quandles, we refer the reader to \cite{FeRo92} and \cite{Shelf-history2}.

\begin{thm}\label{prop-rack-p-r}
    Let $(X, \triangleright)$ be a shelf, $I \subseteq \mathbb{Z}^+$ and, for all $i, j \in I$, let $\alpha_{ij}:X\to X$ be maps. Consider a map $\alpha: I \times I \times X \to X, \, (i, j, a) \mapsto \alpha_{ij}(a)$ such that, for all $a,b \in X$, $i, j, h \in I$,
    \begin{align}\label{cond}
        \alpha_{ih}(a)\triangleright \alpha_{jh}(b) = \alpha_{jh}(\alpha_{ij}(a)\triangleright b).
    \end{align} 
    Define, for all $a, b \in X$, $i, j \in I$,  the binary operation $a \triangleright_{z_{ij}} b := \alpha_{ij}(a) \triangleright b$, then
 the pair $\left(X, \triangleright_{z_{ij}}\right)$ is a p-shelf. Moreover, if $(X, \triangleright)$ is a rack, then $(X,  \triangleright_{z_{ij}})$ is a $p$-rack.
 \begin{proof} 
    The proof is straightforward.
 \end{proof}

\end{thm}

\begin{rem}\label{rem_alpha}
    Note that if the maps $\alpha_{ij}$ are endomorphisms of the shelf $(X, \triangleright)$ such that 
  \begin{align}\label{eq:comp_alpha}
        \forall\ i,j, h \in I\qquad \alpha_{ij} = \alpha_{hj}\alpha_{ih},
    \end{align} 
    then  \eqref{cond} is satisfied. 
    Conversely, if $\alpha_{ij}$ are bijective maps that satisfy \eqref{cond} and \eqref{eq:comp_alpha}, then $\alpha_{ij}$ are endomorphisms of $(X, \triangleright)$.
\end{rem}

We present some examples of $p$-shelves obtained as in \cref{prop-rack-p-r}.
\begin{exa}$ $ \label{shelves}
\begin{enumerate}
    \item[$(1)$] Let $X$ be a set, $h:X \to X$ an idempotent map, and consider the shelf $(X, \triangleright)$ defined by $a \triangleright b:=h(a)$, for all $a, b \in X$. Let $Y \subseteq X$ and, for all $z_{i,j} \in Y$, consider the map $\alpha_{ij}: X \to X$ of constant value $h\left(z_j\right)$. Then, the maps $\alpha_{ij}$ satisfy \eqref{cond} and if we define, for all $a, b \in X$, $$a \triangleright_{z_{ij}} b =h\left(z_j\right),$$ then the pair $\left(X , \triangleright_{z_{ij}}\right)$ is a p-shelf. 
    \item[$(2)$] Let $(X, +)$ be a group and consider the core quandle $(X, \triangleright)$ defined by $a \triangleright b:=a-b+a$, for all $a, b \in X$. Then, for all $z_{i,j} \in Y \subseteq X$, the maps $\alpha_{ij}: X \to X, \, a \mapsto a+z_i-z_j$ satisfy the assumptions in \cref{rem_alpha}. Thus, $\left(X\, ,\triangleright_{z_{ij}} \right)$ is a $p$-rack, with $$a \triangleright_{z_{ij}} b=a+z_i-z_j-b+a+z_i-z_j,$$
    for all $a, b \in X$ and $z_{i, j} \in Y$.
    \item[$(3)$] Let $\left(X, +\right)$ be a group and $f$ a heap endomorphism of $\left(X, +\right)$ 
    or a metahomomorphism of $\left(X, +^{op}\right)$. Then, the structure $\left(X,\triangleright\right)$ with $a \triangleright b:=-f(a)+f(b)+a$, for all $a, b \in X$, is a shelf (for more details, see \cite[Example 3.14]{DoRySt}). Let $\alpha_{ij}:X\to X$ be an endomorphism of $\left(X, \triangleright\right)$ such that, for all $i, j, h\in I$, $\alpha_{ij}=\alpha_{hj}\alpha_{ih}$ (for instance, an endomorphism of $(X, +)$ such that $f\alpha_{ij} = \alpha_{ij}f$ and $\alpha_{ij}=\alpha_{hj}\alpha_{ih}$, for all $i, j, h\in I$). Then, by \cref{rem_alpha}, the structure $\left(X, \triangleright_{z_{ij}}\right)$ with, for all $a, b \in X$ and $i, j \in I$,
    \begin{align*}
       a \triangleright_{z_{ij}} b=-f\left(\alpha_{ij}(a)\right)+f(b)+\alpha_{ij}(a)
    \end{align*}
 is a $p$-shelf that is a $p$-rack if and only if $f$ is bijective. 
\end{enumerate}

\end{exa}

It is clear that \cref{prop-rack-p-r} cannot provide parametric $p$-shelves starting from the \emph{trivial rack} $\left(X, \triangleright\right)$ defined by $a\triangleright b= b$, for all $a,b\in X$.
The following is an example of $p$-shelf that could be obtained by deforming the trivial one, but not as in \cref{prop-rack-p-r}.
\begin{exa}\label{ex_semigroup}
     Let $X$ be a semigroup, $Y \subseteq X$, and set $a \triangleright_{z_{ij}} b:= b z_j$, for all $a, b \in X$ and $z_{i,j} \in Y$. Then $\left(X, \triangleright_{z_{ij}}\right)$ is a $p$-shelf. 
\end{exa}

The following provides a simple way to obtain $p$-shelves, including also the example in \cref{ex_semigroup} and deformations of the trivial rack.
\begin{pro}\label{prop-inv}
    Let $I \subseteq \mathbb{Z}^+$ and, for all $i, j \in I$, let $\beta_{ij}:X\to X$ be maps. Define, for all $a, b \in X$, $i, j \in I$,  the binary operation $a \triangleright_{z_{ij}} b := \beta_{ij}(b)$. Then, $\left(X, \triangleright_{z_{ij}}\right)$ is a $p$-shelf if and only if 
    \begin{align}\label{eq:prop-inv}
        \beta_{ik}\beta_{jk} = \beta_{jk}\beta_{ik},
    \end{align}
    for all $i,j,k\in I$. In particular, $(X, \triangleright_{z_{ij}})$ is a $p$-rack if and only if $\beta_{ij}$ is bijective, for all $i,j\in I$.
\end{pro}
\begin{proof} 
    The proof is straightforward.
\end{proof}

\begin{exa} $ $
    \begin{enumerate}
    \item[$(1)$] Let $X$ be a semigroup and $Y \subseteq X$. Then the maps $\beta_{ij}:X \to X$, $a \mapsto az_{j}$ trivially satisfy \eqref{eq:prop-inv}, hence the binary operation $a \triangleright_{z_{ij}} b=bz_j$ provides the $p$-shelf in \cref{ex_semigroup}. 
    \item[$(2)$] Let $\left(X, +\right)$ be a semigroup and $Y\subseteq X$ such that $z_{i} + z_{j} = z_{j} + z_{i}$, for all $z_{i, j}\in Y$. Then the maps $\beta_{ij}:X\to X, a\mapsto a + z_{i} + z_{j}$, for every $a\in X$, satisfy \eqref{eq:prop-inv} and if we define, for all $a,b\in X$,     
    $a\triangleright_{z_{ij}} b = \beta_{ij}\left(b\right) = b + z_{i} + z_{j}$, we obtain that $\left(X, \triangleright_{z_{ij}}\right)$ is a $p$-shelf that is a $p$-rack if $\left(X, +\right)$ is a group.
    \end{enumerate}  
\end{exa}

The examples of $p$-racks presented next are obtained using the algebraic structures of skew braces. We recall the definition of skew braces below (see \cite{Ru05, GuaVen}).
\begin{defn}\label{skew_defn}
A \emph{skew left brace} is a set $B$ together with two group operations $+,\circ :B\times B\to B$ such that for all $ a,b,c\in B$,
\begin{equation}\label{def:dis}
a\circ (b+c)=a\circ b-a+a\circ c,
\end{equation}
or equivalently,
\begin{equation}\label{def:dis2}
a\circ (b+c)=a\circ b + a\circ\left(a^{-1} + c\right),
\end{equation}
where $a^{-1}$ denotes the inverse of $a$ with respect to the $\circ$ operation. If $(B, +)$ is an abelian group, then $B$ is called 
\emph{left brace}. As usual, the first operation is called \emph{addition} and the second one is called \emph{multiplication}.
\end{defn}
From now on, whenever we say skew brace we mean a skew left brace. Moreover, it is easy to check that the additive and multiplicative identities of a skew brace $B$ coincide and so it will be denoted by $0$.

Clearly, every group $(B, +)$ gives rise to two skew braces $(B, +, +)$ and $\left(B, +, +^{op}\right)$, called the \emph{trivial} and \emph{almost trivial skew brace}, respectively. Some more examples of braces are presented below. 
\begin{exa}\label{ex:fractions} $ $
\begin{enumerate}
    \item[$(1)$] Let $\mathrm{Odd}:= \big\{\frac{2n+1}{2k+1}\ |\ n,k\in\mathbb{Z}\big\}$ together with two binary operations defined by, for all $a, b \in \mathrm{Odd}$,
$a \ {+}_{1} \ b:= a-1+b$
    and $a \ {\circ}_{1} \ b:= a\cdot b$,
 where $+$ and $\cdot$ are the addition and the multiplication 
of rational numbers, respectively. Then, the triple $(\mathrm{Odd},+_1,\circ)$ is a brace. 
\vspace{1mm}
\item[$(2)$] Let $\mathrm{U}(\mathbb{Z}/2^n\mathbb{Z})$ denote the set of invertible integers modulo $2^n$, for some $n\in \mathbb{N}$. 
Then, the triple $(\mathrm{U}\left(\mathbb{Z}/2^n\mathbb{Z}),+_1,\circ\right)$ is a brace where, for all
$a,b\in \mathrm{U}(\mathbb{Z}/2^n\mathbb{Z})$, $a+_1b:=a-1+b$, with $+$ and $\circ $ the addition and multiplication of integer numbers modulo $2^n$, respectively. 
\end{enumerate}
\end{exa}

The following are examples of $p$-racks obtained using the construction method provided in \cref{prop-rack-p-r}. 
\begin{exa}\label{exs_p_shelves}
Let $(X, +, \circ)$ be a skew brace and $(X,\, \triangleright)$ the conjugation quandle, that is, the quandle given by $a \triangleright b = -a + b + a$, for all $a, b \in X$.
Consider $Y\subseteq X$ such that, for all $a,b \in X,$ $z\in Y,$ $(a+b)\circ z = a \circ z -z +b \circ z,$ and every $z \in Y$ is central in $(X, +).$
Besides, for  all $z_{i,j}\in Y$,  the map $\alpha_{ij}:X\to X, \, a\mapsto a \circ z_i\circ z_j^{-1}$ satisfies \eqref{cond} and if we define, for all $a, b \in X$ and $z_{i,j} \in Y$, $$a \triangleright_{{ij}} b =-a\circ z_i \circ z_j^{-1} + b + a \circ z_i \circ z_j^{-1},$$ then $\left(X\, , \triangleright_{z_{ij}} \right)$ is a $p$-rack. 
\end{exa}

\begin{rem}
Let $(X, +, \circ)$ be a skew brace. Note that the set $Y$ in \cref{exs_p_shelves} is a subset of $\mathcal{D}(X) \cap Z(X, +)$, where $Z(X, +)$ is the center of the group $(X, +)$ and
\begin{align*}
    \mathcal{D}(X)=\{z \in X \, \vert \, \forall\ a,b \in X \, \, (a+b) \circ z= a \circ z- z + b \circ z\}
\end{align*}
is the \emph{right distributor of $X$}, introduced in \cite{MaRySt24}, that is a subgroup of $(X, \circ)$.
Let us note that if $Y\neq\emptyset$, since $Z(X,+)$ is a subgroup of the group $(X,\,\circ)$, then $\left(Y,\,\circ\right)$ is a group.
\end{rem}





More generally, we can derive the following example. 

\begin{exa} \label{ex_rack_skew}
    Let $(X,+, \circ)$ be a skew brace, $Y \subseteq \mathcal{D}(X) \cap Z(X, +)$, and $\xi \in X$. 
     Define for  all $z_{i,j}\in Y$, $a,b \in X,$
$$a \triangleright_{z_{ij}} b := -\xi \circ a\circ z_i \circ z_j^{-1}  + \xi \circ b + a \circ z_i \circ z_j^{-1}.$$ 
Then, $\left(X, \triangleright_{z_{ij}}\right)$ is a $p$-rack belonging to the class of $p$-racks in $(3)$  of \cref{shelves} (by choosing $f:X\to X$ the map given by $f\left(a\right) = \xi\circ a$, for all $a \in X$, and $\alpha_{ij}(a)=a \circ z_i \circ z_j^{-1}$).
\end{exa}
Evidently, the $p$-rack in \cref{exs_p_shelves} belongs to the class of $p$-racks in \cref{ex_rack_skew}, by choosing $\xi=0$.

\section{Reflections from \texorpdfstring{$p$-}{}shelves and \texorpdfstring{$p$-}{}racks}

The following results are consistent with those obtained in \cite{AlMaSt24x} in the non-parametric case.
\begin{pro}\label{lemma_generale}
    Let $(X, \triangleright_{z_{ij}})$ be a $p$-shelf, $Y\subseteq X$, and $ R^{z_{ij}}: X \times X \to X \times X$ the $p$-shelf solution defined by $R^{z_{ij}}(b, a) =(b, b \triangleright_{z_{ij}} a)$, for all $a, b \in X$ and $z_{i,j} \in Y$. Let also $K^{z}: X\to X$ such that $a\mapsto \kappa^z(a),$ for all $z \in Y.$ Then $K^z$ is a  reflection if and only if 
\begin{enumerate}
    \item[$(i)$] $\big(\kappa^{z_i}(b\triangleright_{z_{\bar{j}\bar{i}}} a)\big)\triangleright_{z_{ij}} \kappa^{z_j}(b)
    = \kappa^{z_j}\big((\kappa^{z_i}(b\triangleright_{z_{\bar{j}\bar{i}}} a))\triangleright_{z_{i\bar{j}}}b\big)$,
    \vspace{1mm}
    
    \item[$(ii)$] $\kappa^{z_i}((\kappa^{z_j}(b))\triangleright_{z_{j\bar{i}}} a) = \kappa^{z_i}(b\triangleright_{z_{\bar{j}\bar{i}}} a)$,
\end{enumerate}
for all $a,b\in X,$ $z_{i,j}\in Y$.
\begin{proof} The proof is direct by equating the left and right hand sides of the reflection equation \eqref{reflec_eq}.
\end{proof}
\end{pro}

\begin{pro}\label{shelvesf}
Let $(X, \triangleright_{z_{ij}})$ be a $p$-shelf, $Y\subseteq X$, and $ R^{z_{ij}}: X \times X \to X \times X$ the $p$-shelf solution defined by $R^{z_{ij}}(b, a) =(b, b \triangleright_{z_{ij}} a)$, for all $a, b \in X$ and $z_{i,j} \in Y$. Let also $K^{z}:X\to X$ a bijective map such that $a\mapsto \kappa^z(a),$ for all $z \in Y.$ Then $K^z$ is a reflection if and only if 
\begin{enumerate}
    \item[$(i)$] $~\big(b \triangleright_{z_{\bar ji}} a \big)\triangleright_{z_{ij}} \kappa^{z_j}(b) = \kappa^{z_j}\big(\big(b \triangleright_{z_{\bar ji}} a \big) \triangleright_{z_{i\bar j}} b\big)$,
    \vspace{1mm}
    
    \item[$(ii)$] $~\kappa^{z_i}(a) \triangleright_{z_{ ij}} b =a \triangleright_{z_{\bar i j}} b,$
\end{enumerate}
for all $a,b\in X,$ $z_{i,j}\in Y$.
\begin{proof}
   The proof is a direct consequence of \cref{lemma_generale}. 
\end{proof}
\end{pro}

\begin{cor}\label{rackf}
Let $(X, \triangleright_{z_{ij}})$ be a $p$-rack 
$Y\subseteq X$, and $ R^{z_{ij}}: X \times X \to X \times X$ the $p$-rack solution defined by $R^{z_{ij}}(b, a) =(b, b \triangleright_{z_{ij}} a)$, for all $a, b \in X$ and $z_{i,j} \in Y$. Let also $K^{z}: X\to X$ a bijective map such that $a\mapsto \kappa^z(a),$ for all $z \in Y.$ Then $K^z$ is a reflection if and only if 
\begin{enumerate}
    \item[$(i)$] $~a \triangleright_{z_{ij}} \kappa^{z_j}(b) = \kappa^{z_j}(a \triangleright_{z_{i\bar j}} b)$,
    \item[$(ii)$] $~\kappa^{z_i}(a) \triangleright_{z_{ ij}} b =a \triangleright_{z_{\bar i j}} b,$
\end{enumerate}
for all $a,b\in X,$ $z_{i,j}\in Y$.
\begin{proof}
   The proof is a direct consequence of \cref{shelvesf}, using that $L_a^{z_{ij}}$ is a bijection.
\end{proof}
\end{cor}

\smallskip

\begin{exa}\label{Example-refl} 
Let $(X, +, \circ)$ be a skew brace, $Y\subseteq \mathcal{D}(X)\cap Z(X, +)$,  and $\mu: Y\to Y$ the map given by $\mu(z_i) = z_i^{-1}$, for all $z_i\in Y$, hence we write $z_{\bar i } = z_i^{-1}$.
Consider the $p$-rack solution $R^{z_{ij}}$ for the $p$-rack in \cref{ex_rack_skew},
i.e., $a\triangleright_{z_{ij}} b = -\xi \circ a \circ z_i \circ z_j^{-1} + \xi\circ b  + a \circ z_i \circ z_j^{-1}$, for all $a, b \in X, z_{i,j} \in Y$. Let also $\zeta \in Z(X, +)$ such that  $\xi \circ \left(m\zeta\right) =\xi+ m\zeta,$ with $m \in \mathbb{Z}$. 
Then, for all $z \in Y$, the map $K^{z}:X \to X$ defined by $\kappa^z(a) = a\circ z^{-1}\circ z^{-1} +m\zeta$, for all $a \in X$, is a bijective reflection.
\end{exa}

In the following, we provide reflections for the $p$-shelves obtained as in \cref{prop-rack-p-r}. For the remaining of the subsection we introduce the short-hand notation $f_i$ for any map $f_i: X \to X$ that depends on $z_i \in Y.$

\begin{thm}\label{refl-prop-rack-p-r}
    Let $\left(X, \triangleright_{z_{ij}}\right)$ be a $p$-shelf as in \cref{prop-rack-p-r} with $\alpha_{ij}: X \to X$ bijective maps. Let $t_i, \kappa_i:X\to X$ be bijective maps such that $\kappa_{i} = t_i\,\alpha_{\bar{i}i}$, for all $i \in I$. Then, the map $K_i:X\to X$ such that $a\mapsto\kappa_i\left(a\right)$, for all $i\in I$, is a bijective reflection if and only if
    \begin{enumerate}
        \item[$(i)$] $t_i\left(\left(b\triangleright a\right) \triangleright b\right) = (b \triangleright a) \triangleright t_i\left(b\right)$, 
        \item[$(ii)$] 
        $\alpha_{\bar{i}j}\alpha^{-1}_{\bar{i}i}\left(a\right)\triangleright b = \alpha_{ij}t_i\left(a\right)\triangleright b$,
    \end{enumerate}
    for all $a,b\in X$ and $i\in I$.
    \begin{proof}
      To get the proof, let us use \cref{shelvesf}. If $a, b \in X$, $i, j \in I$,  note that, by the definitions, $\kappa_i(a) \triangleright_{z_{ ij}} b =
\alpha_{ij}(t_i\alpha_{\bar{i}i}(a))\triangleright b$ \ and \
$a \triangleright_{z_{\bar i j}} b
= \alpha_{\bar{i}j}(a)\triangleright b$. Hence, by the bijectivity of $\alpha_{\bar{i}i}$, we obtain the equivalence with condition (ii). Now, observe that, by \eqref{cond},
\begin{align*}
    ~\big(b \triangleright_{z_{\bar ji}} a \big)\triangleright_{z_{ij}} \kappa_j(b) 
    = (\alpha_{\bar{j}j}(b)\triangleright\alpha_{ij}(a))\triangleright(t_j\alpha_{\bar{j}j}(b))
\end{align*}
and
\begin{align*}
~\kappa_j\big(\big(b \triangleright_{z_{\bar ji}} a \big) \triangleright_{z_{i\bar j}} b\big)
&= t_j\alpha_{\bar{j}j}\big(\alpha_{i\bar{j}}(\alpha_{\bar ji}(b)\triangleright a)\triangleright b\big)\\
&= t_{j}\left(\alpha_{ij}\left(\alpha_{\bar ji}(b)\triangleright a\right)\triangleright \alpha_{\bar{j}j}(b)\right ) &\mbox{by \eqref{cond}}\\
& = t_j\left(\left(\alpha_{\bar{j}j}(b)\triangleright \alpha_{ij}(a)\right)\triangleright \alpha_{\bar{j}j}(b)\right)&\mbox{by \eqref{cond}}
\end{align*}
Therefore, by the bijectivity of $\alpha_{ij}$ the claim follows.
\end{proof}
\end{thm}

\begin{cor}\label{cor_rack}
    Let $\left(X, \triangleright_{z_{ij}}\right)$ be a $p$-rack as in \cref{prop-rack-p-r} with $\alpha_{ij}: X \to X$ bijective maps. Let $t_i, \kappa_i:X\to X$ be bijective maps such that $\kappa_{i} = t_i\,\alpha_{\bar{i}i}$, for all $i \in I$. Then, the map $K_i:X\to X$ such that $a\mapsto\kappa_i\left(a\right)$, for all $i\in I$, is a bijective reflection if and only if
    \begin{enumerate}
        \item[$(i)$] $t_i\left(a\triangleright b\right) = a\triangleright t_i\left(b\right)$, 
        \item[$(ii)$] 
        $\alpha_{\bar{i}j}\alpha^{-1}_{\bar{i}i}\left(a\right)\triangleright b = \alpha_{ij}t_i\left(a\right)\triangleright b$,
    \end{enumerate}
    for all $a,b\in X$ and $i\in I$.
\end{cor}

\begin{exa}\label{exa_refle}
    The map $K_i:X \to X$ in \cref{Example-refl} can be constructed as in \cref{cor_rack} with $t_i: X \to X$ given by 
$t_i\left(a\right) = a + m\zeta$, for all $a\in X$, where $\zeta\in Z(X, +)$ and $m\in\mathbb{Z}$, and 
$\alpha_{i j}: X \to X$  such that $\alpha_{i  j}(a)= a \circ z_i\circ z_j^{-1}$,
for all $a \in X$ and $z_{i, j} \in Y$, as in \cref{exs_p_shelves}. In this case, $\alpha_{ij}^{-1} = \alpha_{ji}$ and $\alpha_{ji}\alpha_{hj} = \alpha_{hi}$ (hence $\alpha_{ii} = \id_X$). 
\end{exa}

As an application of \cref{refl-prop-rack-p-r} and \cref{cor_rack}, we show how to obtain reflections on $p$-racks starting from reflections on the starting shelves. In this case, maps satisfying \eqref{eq:comp_alpha} can be involved in this objective.

\begin{pro}\label{prop_reflec_Start_shelf}
 Let $\left(X, \triangleright_{z_{ij}}\right)$ be a $p$-shelf (or $p$-rack) as in \cref{prop-rack-p-r} with $\alpha_{ij}: X \to X$ bijective maps satisfying \eqref{eq:comp_alpha}. Let $\kappa: X \to X$ be a bijective reflection for the shelf (or rack) $(X, \triangleright)$ such that $\kappa\alpha_{ij}=\alpha_{ij}\kappa$, for all $i, j \in I$. Then,  the map $K_i:X\to X$ such that $a\mapsto\kappa \alpha_{\bar{i}i}\left(a\right)$, for all $i\in I$, is a bijective reflection.  
   \begin{proof}
       Initially, by \cite[Theorem 2.1]{AlMaSt24x} 
      and bijectivity, the map $\kappa: X \to X$ is a reflection for $(X, \triangleright)$ if and only if, for all $a, b \in X$,
       \begin{align*}
           a \triangleright b= \kappa(a) \triangleright b \qquad \text{and} \qquad \kappa\left(\left(a \triangleright b\right) \triangleright a\right)=\left(a \triangleright b\right) \triangleright \kappa(a).
       \end{align*}
      Thus, condition $(i)$ of \cref{refl-prop-rack-p-r} is straightforward.
       Moreover, by \eqref{eq:comp_alpha}, we have 
    $\alpha_{\bar{i}j}\alpha_{\bar{i}i}^{-1} = \alpha_{ij}$,
       for all $i, j \in I$, and, by the assumption of $\kappa\alpha_{ij} = \alpha_{ij}\kappa$, condition $(ii)$  of \cref{refl-prop-rack-p-r} is satisfied.
   \end{proof}
\end{pro}

Note that the maps $t_i$ in \cref{exa_refle} are not reflections for the starting rack, so reflection $K^z$ is not obtained as in \cref{prop_reflec_Start_shelf}. However, other examples of reflections for $p$-racks we gave before can be constructed by \cref{prop_reflec_Start_shelf}.

\begin{exa}
\label{ex_refl_core}
Let $\left(X, \triangleright_{z_{ij}}\right)$ be the $p$-rack  in \cref{shelves} (2), with $a \triangleright_{z_{ij}} b=a+z_i-z_j-b+a+z_i-z_j,$
    for all $a, b \in X$, $z_{i,j} \in Y$. 
    The map $\kappa:X\to X$ defined by $\kappa\left(a\right) = a+\zeta$, for all $a \in X$, with $\zeta$ an element of order $2$ in $Z(X,+)$, is a bijective reflection of the core quandle $\left(X, \triangleright\right)$. By \cref{prop_reflec_Start_shelf}, consider the map $K_i: X \to X$ given by $\kappa_i(a)=a-z_i-z_i+ \zeta$, for all $i \in I$ and $a \in X$, and $\mu: Y\to Y,$  $\mu(z_i) = -z_i$, for all $z_i\in Y$,  then $K_i$ is a bijective reflection.
\end{exa}


\section{Admissible Drinfel'd twists}

\noindent Before we discuss the main results concerning the parametric reflection equation,
we first review some recent findings on the construction of admissible parametric Drinfel'd twist for $p$-rack type solutions of the Yang-Baxter equation that give rise to more general solutions of the set-theoretic Yang-Baxter equation of the type \cite{Doikoup}
$R^{z_{ij}}: X \times X \to X \times X,$ such that for all 
$a, b \in X,$  $z_{i,j} \in Y,$
\begin{equation} 
R^{z_{ij}} (b,a) = (\sigma^{z_{ij}}_a(b), \tau_b^{z_{ij}}(a)).\label{genset}
\end{equation}
We recall the conditions satisfied by the general set-theoretic solution of the parametric Yang-Baxter equation \cite{Doikoup}. 
\begin{pro} \label{bulk}
Let $X,\ Y\subseteq X,$ be non-empty sets, and define for all $a,b\in X,$ $z_{i,j} \in Y,$  the maps $\sigma_a^{z_{ij}}, \tau_b^{z_{ij}}: X \to X,$ 
$b \mapsto \sigma_a^{z_{ij}}(b)$ and $a \mapsto \tau_b^{z_{ij}}(a).$ 
Then $R^{z_{ij}}:X \times X \to X \times X$, 
such that  for all $a,b\in X,$ $z_{i,j}\in Y,$ $R^{z_{ij}}(b,a) = (\sigma_a^{z_{ij}}(b), \tau_b^{z_{ij}}(a))$ is a solution of the parametric set-theoretic Yang-Baxter equation if and only if, for all $a,b,c \in X$ and $z_{1,2,3} \in Y,$
\begin{eqnarray}
&&  \sigma^{z_{13}}_a\left(\sigma^{z_{12}}_b\left(c\right)\right) = \sigma^{z_{12}}_{\sigma^{z_{23}}_a\left(b\right)}\left(\sigma^{z_{13}}_{\tau^{z_{23}}_b\left(a\right)}\left(c\right)\right)\label{1}\\[0.3cm]
&& \tau^{z_{13}}_c\left(\tau^{z_{23}}_b(a)\right) =\tau^{z_{23}}_{\tau^{z_{12}}_c\left(b\right)}\left(\tau^{z_{13}}_{\sigma^{z_{12}}_b\left(c\right)}(a)\right)\label{2}\\[0.2cm]
&& \sigma^{z_{23}}_{\tau^{z_{13}}_{\sigma^{z_{12}}_b\left(c\right)}\left(a\right)}
    \left(\tau^{z_{12}}_c\left(b\right)\right)
   = \tau^{z_{12}}_{\sigma^{z_{13}}_{\tau^{z_{23}}_b\left(a\right)}\left(c\right)}\left(\sigma^{z_{23}}_a\left(b\right) \right). \label{3} \end{eqnarray}
\end{pro}

\subsection{Twisted reflection maps}
We first derive the main conditions satisfied by solutions of the reflection equation associated with the general-type set-theoretic solution (\ref{genset}).
\begin{pro} \label{bound}
Let $X, Y \subseteq X$, be non-empty sets and $R^{z_{ij}}: X \times X \to X \times X$ a solution of the set-theoretic Yang-Baxter equation, $R^{z_{ij}}(b, a) =(\sigma_a^{z_{ij}}(b), \tau_b^{z_{ij}}(a)).$ Let also $K^{z}: X\to X,$ such that $a\mapsto \kappa^z(a),$ for all $z \in Y.$ Then $K^z$ is a reflection if and only if \begin{enumerate}
    \item[$(i)$] $\sigma^{z_{12}}_{\sigma_{a}^{z_{2\bar1}}(\kappa^{z_2}(b))}(\kappa^{z_1}(\tau^{z_{2\bar 1}}_{\kappa^{z_2}(b)}(a))  =\sigma^{z_{1\bar 2}}_{\sigma_{a}^{z_{\bar 2\bar1}}(b)}(\kappa^{z_1}(\tau^{z_{\bar 2\bar 1}}_{b}(a))$,
    \vspace{2mm}
    
    \item[$(ii)$] $~\tau^{z_{12}}_{\kappa^{z_1}\left(\tau^{z_{2\bar 1}}_{\kappa^{z_2}(b)}\left(a\right)\right)}\left(\sigma_a^{z_{2\bar 1}}(\kappa^{z_2}\left(b\right)\right) =\kappa^{z_2}\left(\tau^{z_{1\bar 2}}_{\kappa^{z_1}\left(\tau^{z_{\bar2\bar 1}}_{b}\left(a\right)\right)}(\sigma_a^{z_{\bar 2\bar 1}}\left(b\right)\right),$
\end{enumerate}
 for all $a, b\in X,$ $z_{1,2} \in Y$.  
\end{pro}
\begin{proof} The proof is direct by equating the left and right hand sides of the reflection equation. 
\end{proof}

We  now introduce the notion of a parametric {\it Drinfel'd twist} and review the
results shown in \cite{Doikoup} (see for the non-parametric Yang-Baxter equation
 \cite{Doikou1, 
 DoiRyb22, DoRySt}, \cite{Sol} \cite{Lebed, LebVen}). 

\begin{defn}\label{def:Drinfiso}
Let $(X, R^{z_{ij}})$ and $(X,S^{z_{ij}})$ be solutions of the parametric set-theoretic Yang-Baxter equation.  
We say that a map 
$\varphi^{z_{ij}}: X\times X\to X\times X$ 
is a \emph{Drinfel'd twist} (\emph{D-twist}) if for all $z_{i,j} \in Y,$
$$
\varphi^{z_{ij}}\, R^{z_{ij}} = S^{z_{ij}}\, (\varphi^{z_{ji}})^{(op)},
$$
where $(\varphi^{z_{ji}})^{(op)} = \pi \circ \varphi^{z_{ji}},$ and $\pi: X \times X \to X \times X$ 
is the ``flip'' map, such that for all $x,y\in X,$ $\pi(x,y) = (y,x).$
If $\varphi^{z_{ij}}$ is a bijection we say that $(X,R^{z_{ij}})$ 
and $(X, S^{z_{ij}})$ are {\it $D$-equivalent (via $\varphi^{z_{ij}}$)}, and 
we denote it by $R^{z_{ij}}\cong_D S^{z_{ij}}$.
\end{defn}

The following proposition and theorem and the corresponding proofs can be found in \cite{Doikoup}.
\begin{pro}\label{Th.r->r'}
    Let $(X,R^{z_{ij}})$ be a left non-degenerate solution, such that for all $a,b \in X,$ $z_{i,j} \in Y,$ $R^{z_{ij}}(b,a) = (\sigma^{z_{ij}}_{a}(b), \tau^{z_{ij}}_{b}(a))$ and let $(X,S^{z_{ij}})$ be a solution associated to a $p$-shelf $\left(X, \triangleright_{ij}\right)$ such that for all $a,b \in X,$ $z_{i,j} \in Y,$  $S^{z_{ij}}(b,a) = (b, b \,\triangleright_{z_{ij}} a)$ and $\tau_b^{z_{ij}}(a) : =(\sigma^{z_{ji}}_{\sigma^{z_{ij}}_a(b)})^{-1}\left(\sigma_a^{z_{ij}}(b) \,\triangleright_{z_{ij}} \, a\right).$ 
Then $R^{z_{ij}}$ is $D$-equivalent to $S^{z_{ij}}$.
\end{pro}

\begin{rem} \label{rem1}
In the special case of reversible $R$-matrices we observe from the fundamental relation $R_{21}^{z_{ji}}R_{12}^{z_{ij}} = \id_{X \times X}$ that $\sigma^{z_{ji}}_{\sigma^{z_{ij}}_{a}{(b)}}\left(\tau_b^{z_{ij}}(a)\right) = a,$ which leads to $b\triangleright_{z_{ij}} a =a,$ for all $a,b\in X,$ $z_{i,j}\in Y,$ and hence $S^{z_{ij}} = \id_{X \times X}.$ 
\end{rem}

\begin{defn}\label{admi2}
    Let $(X,\triangleright_{z_{ij}})$ be a $p$-shelf.  We say that a twist $\varphi^{z_{ij}}: 
    X \times X\to X \times X,$ such that 
for, all $a,b\in X,$ $z_{i,j} \in Y,$ $\varphi^{z_{ij}}(a,b):=(a,\sigma^{z_{ji}}_a(b))$ and $\sigma_a^{z_{ij}}$ is a bijection, is an \emph{admissible twist}, if for all $a,b, c\in X$, $z_{i,j,k } \in Y$:
    \begin{enumerate}
        \item[$(1)$] $\sigma^{z_{ik}}_a(\sigma^{z_{ij}}_b(c)) = \sigma^{z_{ij}}_{\sigma^{z_{jk}}_a\left(b\right)}(\sigma^{z_{ik}}_{\tau^{z_{jk}}_b\left(a\right)}(c))$,
        \vspace{1mm}        
   \item[$(2)$]  $\sigma^{z_{ik}}_c(b) \triangleright_{z_{ij}} \sigma^{z_{jk}}_{c}(a) = \sigma^{z_{jk}}_c(b \triangleright_{z_{ij}} a).$
       \end{enumerate} 
\end{defn}
\smallskip

In the following theorem we show that any left non-degenerate solution $\left(X,\, R^{z_{ij}}\right)$ can be expressed in terms of the $p$-shelf $\left(X,\,\triangleright_{z_{ij}}\right)$  and its admissible twist.

\begin{thm}\label{le:lndsol}
    Let $\left(X,\,\triangleright_{z_{ij}} \right)$ be a $p$-shelf and let $\varphi^{z_{ij}}: X \times X\to X \times X$ be a map such that $\varphi^{z_{ij}}(a,b):=(a,\sigma^{z_{ji}}_a(b))$, for all $a,b\in X,$ $z_{i, j} \in Y.$ Then, the map
$R^{z_{ij}}:X\times X\to X\times X$ defined by
    \begin{equation}\label{thm:solform}
        R^{z_{ij}}\left(b, a\right)  
        = \left(\sigma^{z_{ij}}_a\left(b\right),\, (\sigma^{z_{ji}}_{\sigma_a^{z_{ij}}(b)})^{-1}
        (\sigma_a^{z_{ij}}(b)\triangleright_{z_{ij}} a)\right),  
    \end{equation}
      for all $a,b\in X,$ $z_{i,j} \in Y$, is a solution if and only if $\varphi^{z_{ij}}$ is an admissible twist. 
\end{thm}
\begin{proof} The proof is quite involved and is based on the definition of the admissible Drinfel'd twist in \cref{admi2} and the Yang-Baxter conditions (\ref{1})-(\ref{3}) (see \cite{Doikoup} for the detailed proof).
\end{proof}

\begin{rem} \label{admi3} (Reversible case.) Notice that in the special case of $a\triangleright_{z_{ij}} b =b$, for all $a,b \in X$, $z_{i,j} \in Y$,  the solution (\ref{thm:solform}) of Theorem \ref{le:lndsol}
reduces to the easier expression
\begin{equation}\label{prop:solform}
        R^{z_{ij}}\left(b, a\right)  
        = \left(\sigma^{z_{ij}}_a\left(b\right),\, (\sigma^{z_{ji}}_{\sigma_a^{z_{ij}}(b)})^{-1}
        (a)\right ),
    \end{equation}
    for all $a,b\in X,$ $z_{i,j} \in Y$,  and $R^{z_{ij}}$ is reversible.
\end{rem} 
\smallskip

We can now formulate the main statement on reflections obtained from $p$-racks and twists.

\begin{pro}\label{reflectionA}
 Let $(X,R^{z_{ij}})$ be a left non-degenerate solution such that for all $a,b \in X,$ $z_{i,j} \in Y,$ $R^{z_{ij}}(b,a) = (\sigma^{z_{ij}}_{a}(b), \tau^{z_{ij}}_{b}(a))$ and let $(X,S^{z_{ij}})$ be a solution such that for all $a,b \in X,$ $z_{i,j} \in Y,$  $S^{z_{ij}}(b,a) = (b ,b \triangleright_{z_{ij}} a)$ and $\tau_b^{z_{ij}}(a) : = (\sigma^{z_{ji}}_{\sigma^{z_{ij}}_a(b)})^{-1}(\sigma_a^{z_{ij}}(b)\triangleright_{z_{ij}} a).$ 
 Let also for all $z_i\in Y,$ $K^{z_i}: X \to X,$ $x \mapsto \kappa^{z_i}(x),$ where $\kappa^{z_i}$ is a bijective map, be a solution of the reflection equation, i.e.,
 \begin{equation} 
 S_{12}^{z_{ij}}\ K_1^{z_i}\ S_{21}^{z_{j\bar i}}\ K_2^{z_j}= K_2^{z_j}\ S_{12}^{z_{i\bar j}}\ K_1^{z_i} \  S_{21}^{z_{\bar j \bar i}}. \label{refl1} 
 \end{equation}
 Then, $K^{z_i}$ also satisfies,
\begin{equation} 
R_{12}^{z_{ij}}\ K_1^{z_i}\ R_{21}^{z_{j\bar i}}\ K_2^{z_j} = 
K_2^{z_j}\ R_{12}^{z_{i\bar j}}\ K_1^{z_i}\
R_{21}^{z_{\bar j \bar i }},\label{refl22}
\end{equation}
 if for all $a, b \in X,$ $z_{i,j}\in Y,$ 
 \begin{equation}
 \kappa^{z_j}\left(\sigma_a^{z_{\bar j i}}(b)\right) = \sigma_a^{z_{ji}}\left(\kappa^{z_j}(b)\right). \label{basic0}
 \end{equation}
\end{pro}
\begin{proof} Due to Definition \ref{def:Drinfiso} and Proposition \ref{Th.r->r'} 
we recall that $R^{z_{ij}}$ and $S^{z_{ij}}$ are $D$-equivalent, i.e. 
\begin{equation}
S_{12}^{z_{ij}} = \varphi_{12}^{z_{ij}}\ R_{12}^{z_{ij}}\ (\varphi^{z_{ji}}_{21})^{-1}, \label{twist2}
\end{equation}
where $\varphi^{z_{ij}}: X\times X\to X\times X$ such that $\varphi^{z_{ij}}(a,b) = 
(a, \sigma^{z_{ji}}_{a}(b))$ is the admissible twist of Definition \ref{admi2}. 
By means of \eqref{twist2} the reflection equation 
(\ref{refl1}) becomes
\begin{equation} 
\varphi_{12}^{z_{ij}}R_{12}^{z_{ij}} (\varphi^{z_{ji}}_{21})^{-1} K_1^{z_i} 
\varphi_{21}^{z_{j\bar i}} R_{21}^{z_{j\bar i}}(\varphi^{z_{\bar ij }}_{12})^{-1} K_2^{z_j} = 
K_2^{z_j}\varphi_{12}^{z_{i\bar j}} R_{12}^{z_{i\bar j}}(\varphi^{z_{\bar ji}}_{21})^{-1} K_1^{z_i} 
\varphi_{21}^{z_{\bar j \bar i}}R_{21}^{z_{\bar j \bar i }}(\varphi^{z_{\bar i \bar j }}_{12})^{-1}. 
\end{equation}
From the latter expression, we conclude that if  
\begin{equation}
K_2^{z_j}\ \varphi^{z_{i\bar j}}_{{12}} = \varphi^{z_{ij}}_{{12}}\ K_2^{z_j}  \quad \mbox{or, equivalently,} \quad K_1^{z_i}\ \varphi^{z_{j\bar i}}_{{21}} = \varphi^{z_{ji}}_{{21}}\ K_1^{z_i}\label{cond1}
\end{equation}
then $K^{z_i}$ is a solution of the reflection equation (\ref{refl22}). 
Moreover, condition (\ref{cond1}) readily leads to (\ref{basic0}).
\end{proof}
\smallskip

\begin{rem}
Note that in the reversible case, if condition (\ref{basic0}) holds, then conditions (i) and (ii) of Proposition \ref{bound} automatically hold.
\end{rem}

\begin{rem}
  We are basically interested in invertible reflections $K^z.$ Due to Propositions \ref{reflectionA}, \ref{bulk} and \ref{bound} we conclude that if conditions (i) and (ii) of Corollary \ref{rackf} and (\ref{basic0})
    hold,  then every solution of \eqref{refl1} is also a solution of \eqref{refl22}. The next natural question is: do all solutions of \eqref{refl22} coincide with the solutions of \eqref{refl1}?
\end{rem}


\subsection{Solutions from Drinfel'd twists}\label{sol-Drtwist}

\subsection*{A. Reversible case}
\noindent We first focus on the systematic derivation of reversible, set-theoretic solutions of the parametric Yang-Baxter equation by exploiting the existence of an admissible Drinfel'd twist.

The following useful proposition can be now formulated (see also \cite{Ru07, Ru19, DoRy23}). Henceforth, a map $h:Y \times Y\to Y,$ $(z_i,z_j) \mapsto h(z_{ij})$ will be called \emph{symmetric} if $h(z_{ij}) = h({z_{ji}}),$ for all $z_{i,j} \in Y$.
\begin{pro}\label{p1} 
    Let $(X, \circ)$  be a group, $Y \subseteq X$, and consider, for all $z_{i,j} \in Y,$ $a, b \in X, z_{i,j} \in Y$, $\sigma^{z_{ij}}_a, \tau_b^{z_{ij}}: X \to X$ maps such that  $\sigma^{z_{ij}}_a$ is a bijection and the following hold for all $a, b \in X$ \begin{equation}\label{eqq}
    a\circ b = \sigma_a^{z_{ij}}(b) \circ \tau^{z_{ij}}_{b}(a) \quad \text{and} \quad \sigma^{z_{ji}}_{\sigma^{z_{ij}}_a(b)}\left(\tau^{z_{ij}}_b(a)\right) =a. 
\end{equation}
Moreover, define $\bullet_{z_{ij}}: X\times X \to X$ such that $a \bullet_{z_{ji}} b := a\circ (\sigma^{z_{ij}}_a)^{-1}(b)\circ h(z_{ij}),$  for all $a, b \in X$, $z_{i,j} \in Y,$ where $h:Y\times Y \to Y$ is a symmetric function.
\begin{enumerate}
    \item[$(1)$] Then $a\bullet_{z_{ji}} b = b \bullet_{z_{ij}} a,$ for all $a,b \in X$ and $z_{i,j} \in Y.$
    \item[$(2)$]
     Assume also $(X, + ,\circ)$ is a brace, $Y \subseteq {\cal D}(X)$, and for all $z_{i,j}\in Y,$  $z_i\circ z_j = z_j \circ z_i$.
     If, for all $a,  b \in X$, for all $z_{i,j}\in Y$, we set  $a\bullet_{z_{ij}}b := a \circ z_i + b \circ z_j,$ $h(z_{ij}) = z_i \circ z_j,$ then
     \begin{enumerate}
         \item[$(a)$] $~\sigma_{a}^{z_{ij}}(b) = z_i^{-1} - a \circ z_i^{-1} \circ z_j + a \circ b\circ z_j.$
         \item[$(b)$] $~\sigma^{z_{ik}}_a\left(\sigma^{z_{ij}}_b\left(c\right)\right) = \sigma^{z_{ij}}_{\sigma^{z_{jk}}_a\left(b\right)}\left(\sigma^{z_{ik}}_{\tau^{z_{jk}}_{b}(a)}(c)\right),$ 
     i.e. $\sigma^{z_{ij}}_a$ provides an admissible twist.
     \end{enumerate}
\end{enumerate}
\begin{proof}\hspace{1mm}  
\begin{enumerate}   
\item 
From conditions in \eqref{eqq}, we obtain that $\sigma^{z_{ij}}_{a}(b)^{-1} \circ a \circ b = (\sigma^{z_{ji}}_{\sigma^{z_{ij}}_a(b)})^{-1}(a)$, for all $z_{i,j}\in Y$, $a,b \in X$, that implies
\begin{align}\label{tipo_affine}
a \circ (\sigma^{z_{ij}}_a)^{-1}(b) = b \circ (\sigma^{z_{ji}}_b)^{-1}(a),
\end{align}
for all $z_{i,j}\in Y$, $a,b \in X$, and so $a\bullet_{z_{ji}}b = b\bullet_{z_{ij}}a$.

\item  
$(a)$ For all $a,b \in X$, $z_{i, j} \in Y$, it follows that
\begin{align*}
a\bullet_{z_{ji}}\sigma^{z_{ij}}_a(b) =  a \circ b \circ z_i \circ z_j \ &\Rightarrow \ a \circ z_j + \sigma_a^{z_{ij}}(b)\circ z_i =  a\circ b \circ z_i \circ z_j\  \\ & \Rightarrow\sigma_{a}^{z_{ij}}(b) = z_i^{-1} - a \circ z_i^{-1} \circ z_j + a \circ b\circ z_j.\nonumber
    \end{align*}
$(b)$ It immediately follows from the explicit form of the map $\sigma^{z_{ij}}_{a}(b)$ above.
\hfill \qedhere
\end{enumerate}
    \end{proof}
        \end{pro}
We conclude from Proposition \ref{p1} that the maps $\sigma^{z_{ij}}_{a}(b)$ and $\tau^{z_{ij}}_b(a) =\sigma_a^{z_{ij}}(b)^{-1} \circ a \circ b$  
provide a reversible solution $R^{z_{ij}} (b, a) = (\sigma^{z_{ij}}_{a}(b),\tau^{z_{ij}}_b(a))$ of the parametric set-theoretic braid equation.

\begin{rem}\label{rem42} {(Reflection map.)} 
Let $(X,+, \circ)$ be a brace, $Y = {\cal D}(X)$, and $z_i \circ z_j=z_j\circ z_i$, for all $z_{i,j}\in Y$. Consider the reversible solution $R^{z_{ij}}(b,a)=\left(\sigma_a^{z_{ij}}(b), \tau_b^{z_{ij}}(a)\right)$  associated with $(X,+, \circ)$ where, for all $a,b\in X$ and $z_{i,j}\in Y$, $\sigma_a^{z_{ij}}(b) =z_i^{-1}-a\circ z_i^{-1} \circ z_j+ a\circ b\circ z_j $ and $\tau_b^{z_{ij}}(a) = \sigma_a^{z_{ij}}(b)^{-1}\circ a \circ b.$ 
By \cref{reflectionA}, it follows that the map $K^{z_i}: X \to X$ such that $a \mapsto \kappa^{z_i}(a) : = a \circ z_i^{-1} \circ z_i^{-1} \circ \zeta -\zeta,$ for all $a \in X$ and $z_i \in Y$, where $\zeta \in Y$, and $\mu: Y\to Y,$ $\mu(z_i) = z_i^{-1}$, for all $z_i\in Y$,  is a solution of the reflection equation for the solution $R^{z_{ij}}$ above 
if $\zeta\circ z_i = z_i + \zeta$, for all  $z_i \in Y$. 
\end{rem}

\begin{rem}\label{rem-integrability} ({\bf Solutions from integrability}) We present some known examples of reversible parametric set-theoretic solutions coming from the refactorization of Lax operators (see e.g.  \cite{Adler, Papa2, Veselov}) and extract the associated binary operation $\bullet_{z_{ij}}$ for these solutions. In fact, we focus on families of reversible solutions presented in \cite{Papa2} that satisfy the structure group relation given by $a\circ b = \sigma^{z_{ij}}_a(b) \circ \tau^{z_{ij}}_b(a),$ where $(X, \circ)$ is a group ($X = \mathbb{CP}^1$ for the following solutions):
\begin{enumerate}
   \item[$(1)$] $\sigma^{z_{ij}}_a(b) =a  - P_{z_{ij}}(a,b)$ and $\tau_b^{z_{ij}}(a) = b+P_{z_{ij}}(a,b),$ where $P_{z_{ij}}(a,b) = \frac{z_i -z_j}{a +b}.$\\
In this case, $a\circ b = a+b.$ We derive $(\sigma_a^{z_{ij}})^{-1}(b) = -a + \frac{z_i-z_j}{a-b}$ and hence $a\bullet_{z_{ij}}b =\frac{z_j-z_i}{a-b}.$ We also observe that $a\bullet_{z_{ij}} b = b \bullet_{z_{ji}} a$ as expected. 

$ $

   \item[$(2)$]  $\sigma^{z_{ij}}_a(b) =a\ Q_{z_{ij}}^{-1}(a,b)$ and $\tau_b^{z_{ij}}(a) = b\ Q_{z_{ij}}(a,b),$ where $Q_{z_{ij}}(a,b) = \frac{z_iab +1}{z_ja b+1}.$\\ In this case, $a\circ b = ab$. We derive $(\sigma_a^{z_{ij}})^{-1}(b) =a^{-1} \frac{b-a}{z_ja -z_ib}$ and hence $a\bullet_{z_{ij}} b = \frac{b-a}{z_ia -z_jb},$ also
$a\bullet_{z_{ij}} b = b \bullet_{z_{ji}} a.$    

$ $
   
   \item[$(3)$] $\sigma^{z_{ij}}_a(b) =a\ Q_{z_{ij}}^{-1}(a,b)$ and $\tau_b^{z_{ij}}(a) = b\ Q_{z_{ij}}(a,b),$ where $Q_{z_{ij}}(a,b) = \frac{z_i +(z_j -z_i)a - z_j ab }{z_j + (z_i-z_j) b-z_i ab}.$\\
In this case $a\circ b = ab$. We obtain that $(\sigma_a^{z_{ij}})^{-1}(b) =a^{-1} \frac{z_i b +(z_j-z_i)ab-z_ja}{z_jb  +(z_i -z_j)-z_ia}$ and hence $a\bullet_{z_{ij}} b = \frac{z_j b +(z_i-z_j)ab-z_ia}{z_ib  +(z_j -z_i)-z_ja},$ also
$a\bullet_{z_{ij}} b = b \bullet_{z_{ji}} a.$       

$ $
   
   \item[$(4)$] $\sigma^{z_{ij}}_a(b) =a\ Q_{z_{ij}}^{-1}(a,b)$ and $\tau_b^{z_{ij}}(a) = b\ Q_{z_{ij}}(a,b),$ where $Q_{z_{ij}}(a,b) = \frac{(1-z_j)ab +(z_j-z_i) a +z_j(z_i-1)}{(1-z_i)ab +(z_i-z_j) b +z_i(z_j-1)}.$\\
   In this case $a\circ b = ab$.    We derive $(\sigma_a^{z_{ij}})^{-1}(b) =a^{-1} \frac{z_j(z_i-1)b+(z_j-z_i)ab-  z_i(z_j-1)a}{(1-z_i)a  +(z_i -z_j)-(1-z_j)b}$ and hence $a\bullet_{z_{ij}} b =  \frac{z_i(z_j-1)b+(z_i-z_j)ab-  z_j(z_i-1)a}{(1-z_j)a  +(z_j -z_i)-(1-z_i)b},$ also
$a\bullet_{z_{ij}} b = b \bullet_{z_{ji}} a.$       \end{enumerate}

Reflection maps for set-theoretic solutions from integrability were studied in \cite{CauCra, Cau2}.
The properties of $(X, \bullet_{z_{ij}}, \circ)$ associated to the above solutions merit further investigation, which however will be undertaken in future works. Note however, that in the subsequent section we introduce algebraic structures $(X, +_{z_{ij}}, \circ)$ (where $(X,\circ )$ is a group) that generalize the known (skew) braces and are called (skew) $p$-braces.
\end{rem}

\subsection*{B. General case}
\noindent We recall from Section 3 (see also \cite{DoRySt}) that set-theoretic solutions, which are the main focus of this subsection, are constructed from the $p$-rack solutions via an admissible Drinfel'd twist.

The following proposition is useful for describing general set-theoretic solutions.
\begin{pro} \label{p2} 
Let $\left(X, \triangleright_{z_{ij}}\right)$ be a $p$-shelf, $\sigma^{z_{ij}}_a, \tau_b^{z_{ij}}: X \to X$ maps such that for all $a,b\in X,$ $z_{i,j}\in Y \subseteq X,$ $\sigma^{z_{ij}}_a$ is a bijection and $\tau^{z_{ij}}_{b}(a) = (\sigma^{z_{ji}}_{\sigma^{z_{ij}}_a(b)})^{-1}(\sigma^{z_{ij}}_a(b)\, \triangleright_{z_{ij}} a).$  Moreover, let $(X, \circ)$ be a group such that, for all $a,b \in X$ $z_{i,j}\in Y$,
$a\circ b = \sigma^{z_{ij}}_a(b) \circ \tau^{z_{ij}}_{b}(a)$ and define $\bullet_{z_{ij}}: X\times X \to X$ by $a \bullet_{z_{ji}} b := a\circ (\sigma^{z_{ij}}_a)^{-1}(b)\circ h(z_{ij}),$ for all $a, b \in X$, $z_{i,j} \in Y,$ where $h: Y \times Y\to Y$ is a symmetric function.
Then, for all $a,b \in X,$ $z_{i,j} \in Y:$ 
\begin{enumerate}
    \item[$(1)$] $(a)$\ $a \bullet_{z_{ji}} b = b\bullet_{z_{ij}} (b \triangleright_{z_{ij}} a)$,\\ 
$(b)$\  $a\bullet_{z_{ji}} \sigma^{z_{ij}}_a(b) =  a\circ b\circ h({z_{ij}}).$ 
    \item[$(2)$] Assume that $(X,+, \circ)$ is a skew brace, $Y \subseteq {\cal D}(X) \cap Z(X,+)$ and also for all $z_{i,j}\in Y,$  $z_i\circ z_j = z_j \circ z_i$. Set, for all $a,b \in X$ and $z_{i,j} \in Y,$ $a\bullet_{z_{ij}} b := a\circ z_i+b\circ z_j,$ also $h(z_{ij}) =z_i \circ z_j.$  Then, for all $a,b,c \in X$ and $z_{i,j} \in Y, $
    \begin{enumerate} 
    \item[$(a)$]  $~ b\triangleright_{z_{ij}} a =-b\circ z_i \circ z_j^{-1}+  a + b \circ z_i \circ z_j^{-1}$ 
    and $\sigma^{z_{ij}}_a(b) = z_i^{-1}-a\circ z_i^{-1}\circ z_j + a\circ b\circ z_j,$
    \item[$(b)$] $\sigma^{z_{ik}}_a(\sigma^{z_{ij}}_b(c)) = \sigma^{z_{ij}}_{\sigma^{z_{jk}}_a(b)}(\sigma^{z_{ik}}_{\tau^{z_{jk}}_{b}(a}(c))$ and $\sigma^{z_{ik}}_c(b) \triangleright_{z_{ij}} \sigma^{z_{jk}}_c(a) = \sigma^{z_{jk}}_c(b \triangleright_{z_{ij}} a)$ i.e., $\sigma^{z_{ij}}_a$ provides an admissible twist. 
    \end{enumerate}
    \end{enumerate}
    \begin{proof}

$ $
    \begin{enumerate}
        \item  $(a)$ Using the definitions of  $\bullet_{z_{ji}}$ and $\triangleright_{z_{ij}}$, we get
       for all $a,b \in X,$ $z_{i,j} \in Y,$ 
        \begin{equation}
        b \bullet_{z_{ij}} (b \triangleright_{z_{ij}} a) = b \bullet_{z_{ij}}  \sigma^{z_{{ji}}}_b\left(\tau^{z_{ij}}_{(\sigma^{z_{ij}}_a)^{-1}(b)}(a)\right) =    
        b \circ \tau^{z_{ij}}_{(\sigma^{z_{ij}}_a)^{-1}(b)}(a)\circ h(z_{ij}).    \nonumber
        \end{equation}
   But, due to the structure group condition $a\circ b = \sigma^{z_{ij}}_a(b) \circ \tau^{z_{ij}}_{b}(a)$, we conclude that 
   $$a \bullet_{z_{ji}} b = b\bullet_{z_{ij}} (b \triangleright_{z_{ij}} a).$$

\noindent $(b)$ From the definition of $a\bullet_{z_{ij}} b$, and the fact that $\sigma_a$ is bijection, it immediately follows that
\begin{equation}
a\bullet_{z_{ji}} \sigma^{z_{ij}}_a(b) = a\circ \left(\sigma^{z_{ij}}_a\right)^{-1}\left(\sigma^{z_{ij}}_a(b)\right)\circ h(z_{ij}) = a \circ b \circ h(z_{ij}). \nonumber
\end{equation}

\item $(a)$ The first part follows immediately from (1) $(a)$. From (1) $(b)$
 we immediately conclude that $\sigma^{z_{ij}}_a(b) = z_i^{-1} -a\circ z_i^{-1} \circ z_j +a \circ b \circ z_j.$

\noindent $(b)$ The  two conditions follow from the explicit form of $\sigma_a^{z_{ij}}$ above. 
\hfill \qedhere
\end{enumerate}
\end{proof}
    \end{pro}

Notice that the binary operation $\bullet_{z_{ij}}$ such that it satisfies condition (1) $(a)$ in \cref{p2} is not uniquely defined. However, based on the definition of the operation $\bullet_{z_{ij}}$ 
we provide below a classification of bijective set-theoretic solutions given a specific type of $p$-rack.

%


\begin{exa}\label{ex_conj}
(\textbf{Conjugate $p$-rack}) \ 
Let $(X,+, \circ)$ be a skew brace, $Y \subseteq {\cal D}(X) \cap Z(X,+)$, and also for all $z_{i,j}\in Y,$  $z_i\circ z_j = z_j \circ z_i$. Set, for all $a,b \in X$ and $z_{i,j} \in Y$, 
\begin{center}
 $~ a\triangleright_{z_{ij}} b =-a\circ z_i \circ z_j^{-1}+  b + a \circ z_i \circ z_j^{-1}$.    
\end{center}
Then, $\left(X, \triangleright_{z_{ij}}\right)$ is the $p$-rack obtained also in \cref{exs_p_shelves}, that we call \emph{conjugate $p$-rack}.
Recall that the map
$\sigma^{z_{ij}}_a(b) = z_i^{-1}-a\circ z_i^{-1}\circ z_j + a\circ b\circ z_j$ satisfies condition $(2)$ of Definition \ref{admi2}. We also confirm that condition $(1)$ of Definition \ref{admi2} is equivalent to 
$a\circ b = \sigma^{z_{ij}}_a(b) \circ \tau^{z_{ij}}_b(a)$, for all $a, b \in X$, hence the maps $\sigma^{z_{ij}}_a(b)$ provides a solution. 
Note that this corresponds to the solution derived in \cite[Proposition 2.15]{Doikoup}.

By condition \eqref{basic0} in \cref{reflectionA}, for all $z_i \in Y$, the map $K^{z_i}: X \to X$  given by $\kappa^{z_i}(a) := a \circ z_i^{-1} \circ z_i^{-1} + m\zeta$, with $m \in \mathbb{Z}$ and $\zeta \in Z(X,+),$ is a bijective reflection for the general set-theoretic solution  $R^{z_{ij}}(b,a) = (\sigma_a^{z_{ij}}(b), \tau_b^{z_{ij}}(a))$
if and only if 
\begin{equation}
\forall\, a \in X \quad a+m\zeta=a \circ \left(m\zeta\right), \qquad \forall\, z_i \in Y \quad \left(m \zeta\right) \circ  z_i= m\zeta + z_i.  \label{cond0}
\end{equation}
If we relax conditions  \eqref{cond0}, then $K^{z_i},$ as defined above, 
is a reflection only for the $p$-rack solution $S^{z_{ij}}(b, a) = (b, b \triangleright_{z_{ij}} a)$, for all $a,b\in X$, $z_{i,j}\in Y$ (see also \cref{Example-refl}).
\end{exa}

\smallskip

\noindent Below, we provide further examples of set-theoretic solutions of the Yang-Baxter equation and associated reflection maps.

\begin{exa} ({\bf Affine $p$-rack}) 
Let $(X,+, \circ)$ be a skew brace, $Y \subseteq {\cal D}(X) \cap Z(X,+)$, and assume for all $z_{i,j}\in Y,$  $z_i\circ z_j = z_j \circ z_i$. Let also $\xi\in X$ and define $a\bullet_{z_{ji}} b :=   \xi\circ a\circ z_j+ b \circ z_i$,
for all $a,b \in X$. Then, from Proposition \ref{p2} (1) $(a)$, we  obtain
\begin{center}
    $a\triangleright_{z_{ij}} b = - \xi \circ a \circ z_i\circ z_j^{-1} + \xi\circ b +a \circ z_i \circ z_j^{-1}.$
\end{center}
Then  $(X, \triangleright_{z_{ij}})$ is a $p$-rack which we call the \emph{affine $p$-rack} (see also \cref{ex_rack_skew}).
Now, from \cref{p2}, if we consider $h({z_{ij}}) =  \xi \circ z_i \circ z_j,$ for all $z_{i,j} \in Y$, we obtain the map $\sigma^{z_{ij}}_a(b)$ defined by $$\sigma^{z_{ij}}_a(b) =  z_i^{-1} -  \xi\circ a  \circ z_i^{-1} \circ z_j +  a \circ b \circ \xi \circ  z_j,$$ 
for all $a, b \in X$, $z_{i,j} \in Y$, that also satisfies the identity $a\circ b = \sigma^{z_{ij}}_a(b)\circ \tau^{z_{ij}}_b(a)$. If, in addition,
$\xi\in D(X)\cap Z\left(X,\circ \right)$, the map $\sigma^{z_{ij}}_a$ satisfies
conditions (1) and (2) of Definition \ref{admi2}. Hence, by \cref{le:lndsol}, the map $R^{z_{ij}}(b,a) = \left(\sigma_a^{z_{ij}}(b), \tau_b^{z_{ij}}(a)\right)$ is a solution.

Now, assume $\zeta \in Z(X, +),$ such that  $\xi \circ \left(m\zeta\right) =\xi+ m\zeta$, for all $m\in\mathbb{Z}$. 
Then, by condition \eqref{basic0} in \cref{reflectionA}, the map $K^z_i: X \to X$  given by $\kappa^{z_i}(a) := a \circ z_i^{-1} \circ z_i^{-1}+m\zeta$ is a bijective reflection for the general solution $R^{z_{ij}}(b,a) = \left(\sigma_a^{z_{ij}}(b), \tau_b^{z_{ij}}(a)\right)$ if and only if conditions in \eqref{cond0} are satisfied.
If we relax condition \eqref{cond0}, then $K^{z_i}$ as defined above is a reflection 
only for the $p$-rack solution $S^{z_{ij}}(b, a) = (b, b \triangleright_{z_{ij}} a)$ (see also \cref{Example-refl}).
\end{exa}

\begin{exa}
 ({\bf Core $p$-rack})
Let $(X, +, \circ )$ be a skew brace, $Y\subseteq \mathcal{D}(X)$, $z_i\circ z_j  = z_j \circ z_i$, for all  $z_{i,j}\in Y$, and define, for all $a,b \in X,$ $z_{i,j}\in Y$ $$a \triangleright_{z_{ij}} b  
= a \circ z _i\circ z_j^{-1}-b +a\circ z_i \circ z_j^{-1}.$$ Then $(X, \triangleright_{z_{ij}})$ is a $p$-rack, which we call the core \emph{$p$-rack} (note that this example coincides with $(2)$ of \cref{shelves} assuming that, more specifically, $z_i\in Fix(X)$, for all $z_i\in Y$, where $Fix(X):=\{z\in X \mid \forall\ a\in X \ a\circ z = a + z\}\subseteq \mathcal{D}\left(X\right)$). We distinguish two cases: $ $
\begin{enumerate}[$(a)$]
\item We define for all $a,b \in X,$ $z_{i,j} \in Y,$ $a\bullet_{z_{ij}} b := a\circ z_i - b \circ z_j.$ 
Then according to Proposition \ref{p2} for $h(z_{ij}) = z_i \circ z_j $, $$\sigma^{z_{ij}}_a(b) =z_i^{-1} - a \circ b \circ z_j+ a \circ z_i^{-1}\circ z_j.$$
Also, $\sigma^{z_{ij}}_a$ satisfies conditions $(1)$ and $(2)$ of Definition \ref{admi2} 
and so it provides a solution.
$ $
\item  We define for all $a,b \in X,$ $z_{i,j} \in Y,$ $a\bullet_{z_{ij}} b := -a\circ z_i + b \circ z_j.$
Then according to Proposition \ref{p2} for $h(z_{ij}) = z_i \circ z_j $, $$\sigma^{z_{ij}}_a(b) = a\circ z_i^{-1}\circ z_j  -z_i^{-1} + a \circ b \circ z_j.$$

$\sigma^{z_{ij}}_a$ satisfies conditions $(1)$ and $(2)$ of Definition \ref{admi2} and as such it provides a solution. Notice, in particular that $\sigma^{z_{ik}}_a(\sigma^{z_{ij}}_b(c)) = \sigma^{z_{i,jk}}_{a\circ b}(c),$ for all $a, b,c \in X,$ if and only if $(X,+,\circ)$ is a brace; also, recall that $z_{i,jk}$ denotes dependence on $(z_i, z_{j}\circ z_k).$

\end{enumerate}

Now, the map $K^{z_i}: X \to X,$ $a \mapsto \kappa^{z_i}(a) := a \circ z_i^{-1} \circ z_i^{-1} + \zeta,$ with $\zeta \in Z(X,+)$ an element of order $2$ and such that $\zeta\circ z_i = z_i + \zeta$, for all $z_i\in Y$, is a reflection for the $p$-rack solution $S^{z_{ij}}(b, a) = (b,\, b\triangleright_{z_{ij}} a)$
(see also \cref{ex_refl_core}). Then, if $a \circ \zeta=a+\zeta$, for all $a \in X$, by \cref{reflectionA} 
the map $K^{z_i}$ is a reflection for the general solutions $R^{z_{ij}}(b,a) = \left(\sigma_a^{z_{ij}}(b),\, \tau_b^{z_{ij}}(a)\right)$ above.
\end{exa}

\section{\texorpdfstring{$p$-}{}braces and skew \texorpdfstring{$p$-}{}braces}

\noindent Motivated by the results of Subsection \ref{sol-Drtwist}, we generalize the notion of \emph{affine structures} and the braces introduced by Rump \cite{Rump1, Rump2} to the parametric case. We distinguish two cases, the reversible case, where structures similar to braces are introduced, and the general invertible case, where objects similar to semi-affine structures \cite{St23} and skew braces \cite{GuaVen} are derived. The newly derived affine and semi-affine structures are also compatible with the definitions of $p$-set Yang-Baxter operators and groups introduced in Subsections 2.3 of \cite{Doikoup} and also the definitions of decorated $p$-rack algebras and $p$-set Yang-Baxter algebras introduced in Subsection 3.2 of \cite{Doikoup}.

\subsection*{A. Reversible case}
In \cite{Rump1, Rump2}, Rump introduced the notion of \emph{affine structure} on a group to describe braces. The condition \eqref{tipo_affine} arising in the proof of \cref{p1} (1) suggests introducing a similar notion in the parametric case and studying more general algebraic structures similar to braces depending on extra parameters.

\smallskip

We first introduce some useful notation. Let $(X, \circ)$ be a group and $0$ its identity, we then denote by $z_{i,jk}$ (or $i,jk$) the dependence on $(z_i, z_j \circ z_k)$. If we have dependence on two parameters $z_{i,j}$ we then simply write $z_{ij}$ (or $ij$). Also, hereinafter, we set $z_0:= 0$ and $z_{\bar{i}}:= z_{i}^{-1}.$

\begin{defn} \label{def_affine}
    Let  $\left(X,\,\circ\right)$  be a group, $Y\leq X$, and $\eta^{z_{ij}}:X\to Sym(X)$ a
	map from $\left(X,\,\circ\right)$ to the symmetric group on the set $X$ such that, for all $a,b\in X$ and $z_{i,j,k} \in Y$,
	\begin{align}\label{affine1}
	   \eta^{z_{i, jk}}_{a\circ b} = \eta^{z_{ij}}_b\eta^{z_{ik}}_a
       = \eta^{z_{i, kj}}_{a\circ b}.
	\end{align}
 Then, $\eta^{z_{ij}}$ is said a \emph{reversible $p$-affine structure} on  $\left(X,\circ\right)$ if, for all $a,b\in X$ and $z_{i, j}  \in Y$,
	\begin{align}\label{affine2}
	 a\circ\eta^{z_{ij}}_{a}\left(b\right)
	 =  b\circ\eta^{z_{ji}}_{b}\left(a\right).
	\end{align}
\end{defn}
\begin{rem}\label{rem_eta}
Note that if $\eta^{z_{ij}}$ is a reversible $p$-affine structure on a group $(X, \circ)$, then $\eta_0^{z_{i0}} = \eta_{0\circ 0}^{z_{i,00}} = \eta_0^{z_{i0}}\eta_0^{z_{i0}}$, and so $\eta_0^{z_{i0}} = \id_X$.
\end{rem}

In the following, we show that, since affine structures lead to braces, reversible $p$-affine structures provide similar brace-like structures. We call such structures \emph{$p$-braces} and we give the definition below.
\begin{defn}\label{p-brace}
    Let $\left(X,\,\circ\right)$ be a group, $Y\leq X$, and $+_{ij}$ a binary operation on $X$, for all $z_{i,j} \in Y$.
    A triple $\left(X, +_{ij}, \circ\right)$ is a \emph{$p$-brace} if, for all $a,b,c \in X$ and $z_{i, j, k} \in Y$
    \begin{enumerate}
	\item[$(1)$] $\left(a +_{ij} b\right) +_{k,ij} c = a +_{ki,j} \left(b +_{ki} c\right)$,
	\item[$(2)$] $a +_{i,kj} b=a +_{i,jk} b = b +_{jk,i} a $,
    \item[$(3)$] $a = 0 +_{i0} a$,
    \item[$(4)$] $\exists \vert\, x \in X$ such that $a +_{ij}x=0$,
    \end{enumerate}
    and
    \begin{align}\label{prop-p-brace}
        a\circ\left(b+_{k,ij}c\right) = a\circ b+_{ki} a\circ \left(a^{-1} +_{kj} c\right).
    \end{align}   
\end{defn}

    Note that condition \eqref{prop-p-brace}
    is a generalization of the distributivity brace condition \eqref{def:dis2}.

In the following we describe the additive structure of any $p$-brace.
\begin{pro}\label{rem_left_canc}
Let $\left(X, +_{ij}, \circ\right)$ be a $p$-brace and $Y \leq X$. Then for all $z_{i,j}\in Y$ the additive structure $\left(X, +_{ij} \right)$ is a quasigroup. In particular,  $\left(X, +_{i0} \right)$ is a quasigroup with a left identity and $\left(X, +_{0i} \right)$ is a quasigroup with a right identity, for all $z_{i}\in Y$.
\begin{proof}
    Let us first prove that $\left(X, +_{ij} \right)$ is a left cancellative structure. Let $a, b, x\in X$, $z_{i,j} \in Y$, and assume $x +_{ij} a = x +_{ij} b$. If $x'\in X$ is such that $x' +_{j0} x =0$, by $(1)$ and $(3)$, we have that
\begin{align*}
    a &= 0+_{ij,0}a
= (x' +_{j0} x) +_{ij,0} a
= x' +_{ij,0} + (x+_{ij} a)
=  x' +_{ij,0} + (x+_{ij} b)\\
&= (x' +_{j0} x) +_{ij,0} b
= 0+_{ij,0}b
= b.
\end{align*}
Clearly, by $(2)$, $\left(X, +_{ij} \right)$ is also right cancellative. Moreover, let $a,b\in X$ and $z_{i,j}\in Y$. Then, $x:= b +_{\bar{i}, ji} a^-\circ\left(a +_{\bar{i}\,\bar{i}} a^-\right)$ and
$y:= a\circ\left(a^- +_{i,\bar{j}} a^-\circ b\right)$
are such that $x +_{i,j} a = b$ and $a+_{i,j} y = b$. From the first part of the proof,  these elements $x$ and $y$ are unique.
Finally, the last part of the statement follows by (2) and (3) of \cref{p-brace}.
\end{proof}
\end{pro}

In the following theorem, we show that reversible $p$-affine structures yield $p$-braces and reversible solutions of the Yang-Baxter equation.
\begin{thm}\label{thm_affine}
    Let $(X,\,\circ)$, be a group, $Y \leq X$, and $\eta^{z_{ij}}:X\to Sym(X)$ a reversible $p$-affine structure. Define, for all $a,b\in X$ and $z_{i,j} \in Y$, the following binary operation on $X$
\begin{align*}
a+_{ij}b:=            a\circ\eta^{z_{ij}}_{a}\left(b\right).
\end{align*}
Then, the triple $\left(X, +_{ij}, \circ \right)$ is a $p$-brace. In addition, the maps defined by 
\begin{center}
    $\sigma^{z_{ij}}_a(b):= (\eta_{a}^{z_{ij}})^{-1}(b)$ \quad \text{and} \quad $\tau^{z_{ij}}_b\left(a\right):=
\sigma_a^{z_{ij}}\left(b\right)^{-1}\circ a\circ b$, 
\end{center}for all $a, b \in X$ and $z_i,z_j \in Y$, determine a reversible solution $R^{z_{ij}}(b, a)=\left(\sigma^{z_{ij}}_a(b), \tau^{z_{ij}}_b\left(a\right)\right)$ on~$X$.
\end{thm}
\begin{proof}
Let $a,b,c \in X$ and $z_{i,j,k} \in Y$. Then
\begin{align*}
    \left(a +_{ij} b\right) +_{k,ij} c&=b \circ \eta_b^{z_{ji}}(a) +_{k,ij}c &\mbox{by \eqref{affine2}}\\    
        &=b \circ \eta_b^{z_{ji}}(a) \circ \eta_{b \circ \eta_b^{z_{ji}}(a)}^{z_{k, ji}}(c) &\mbox{by \eqref{affine1}}\\
    &=b \circ \eta_b^{z_{ji}}(a) \circ \eta^{z_{kj}}_{\eta_b^{z_{ji}}(a)}\eta^{z_{ki}}_{b}(c)&\mbox{by \eqref{affine1}}\\
    &=b \circ \eta^{z_{ki}}_{b}(c) \circ \eta^{ z_{j,ki}}_{ b \circ \eta^{z_{ki}}_{b}(c)} (a)&\mbox{by \eqref{affine1}}
\end{align*}
and
\begin{align*}
    a +_{ki,j} \left(b+_{ki}c\right)&=a \circ \eta_a^{z_{ki,j}}\left(b \circ \eta_b^{z_{ki}}(c)\right)\\
    &=b \circ \eta^{z_{ki}}_b(c) \circ \eta^{z_{j,ki}}_{b \circ \eta^{z_{ki}}_b(c)}(a) &\mbox{by \eqref{affine2}}
\end{align*}
thus the identity $(1)$ on \cref{p-brace} is satisfied. Moreover, the equality $(2)$ immediately follows from the definition of $+_{i,j}$ and \cref{def_affine}.
Now, by \cref{rem_eta}, $\eta_0^{z_{i0}} = \id_X$, hence $(3)$ holds.
It follows that $a +_{ij} \eta^{z_{i\,\bar{j}}}_{a^{-1}}\left(a^{-1}\right) = 0$ and, by \eqref{affine2} and the bijectivity of the maps $\eta^{z_{i\,j}}_{a}$,  condition $(4)$ is satisfied.
Furthermore, equation \eqref{prop-p-brace} is satisfied since
 \begin{align*}
     a\circ\left(b+_{k,ij}c\right)&=a \circ b \circ \eta_b^{z_{k, ij}}(c)
 \end{align*}
and, by \eqref{affine1},
\begin{align*}
   a\circ b +_{ki} a\circ \left(a^{-1} +_{k,j} c\right)&= a \circ b \circ \eta^{z_{ki}}_{a \circ b}\,\eta_{a^{-1}}^{z_{kj}}(c)=a \circ b \circ \eta_b^{z_{k, ij}}(c).
\end{align*}
Therefore, $\left(X, +_{ij}, \circ \right)$ is a $p$-brace.\\
Besides, again by \eqref{affine1}, it follows also that $(\eta_{a}^{z_{ij}})^{-1} = \eta_{a^{-1}}^{z_{i\bar{j}}}$, for every $a\in X$. 
Thus, 
\begin{center}
 $\sigma^{z_{ij}}_a(b)= (\eta_{a}^{z_{ij}})^{-1}(b) = a \circ \left(a^{-1}+_{i\bar{j}} b\right)$ \quad \text{and} \quad $\tau^{z_{ij}}_b(a)=\sigma_a^{z_{ij}}\left(b\right)^{-1}\circ a\circ b$.    
\end{center}
Hence, $\sigma_{0}^{z_{i0}} = \id_X$ and $(\sigma_{a}^{z_{ij}})^{-1} = \sigma_{a^{-1}}^{z_{i\bar{j}}}$, for every $a\in X$. Thus, if $a, b \in X$, we get
\begin{align*}
a\triangleright_{z_{ij}} b
&=\sigma^{z_{ji}}_{a}\tau^{z_{ij}}_{\left(\sigma^{z_{ij}}_b\right)^{-1}\left(a\right)}\left(b\right) = a\circ\left(a^{-1} +_{j\bar{i}} a^{-1}\circ\left(b+_{ij}a\right)\right)
\underset{(2)}{=} a\circ\left(a^{-1} +_{j\bar{i}} a^{-1}\circ\left(a+_{ji}b\right)\right)\\
&= \sigma_a^{z_{ji}}\sigma_{a^{-1}}^{z_{j\bar{i}}}(b)
= \sigma^{z_{j, \bar{i}i}}_{0}(b) = \sigma^{z_{j0}}_{0}(b)= b,
\end{align*}
namely, $\left(X,\,\triangleright_{z_{ij}}\right)$ is the trivial $p$-rack and so condition \cref{admi2} (2) holds.
Besides, since
\begin{align*}
    \sigma^{z_{ij}}_a(b) \circ \tau^{z_{ij}}_b(a)=a\circ b
\end{align*}
and, by \eqref{affine1} $\sigma^{z_{i, jk}}_{a\circ b} = \sigma^{z_{i, kj}}_{a\circ b}$, for all $a, b \in X$ and $z_{i,j, k} \in Y$, we obtain condition \cref{admi2} (1). Consequently, maps $\sigma_a^{z_{ij}}$ determine an admissible twist, and by \cref{le:lndsol} the claim follows.
\end{proof}


\begin{pro}\label{pro_p_brace}
Let $\left(X, +_{ij}, \circ \right)$ be a $p$-brace and $Y \leq X$. Then, the map $\eta^{z_{ij}}_a:X\to X$, given by $\eta^{z_{ij}}_{a}\left(b\right):= a^{-1}\circ\left(a +_{ij} b\right)$, 
for all $a,b\in X$ and $z_{i,j} \in Y$, defines a reversible $p$-affine structure $\eta^{z_{ij}}$ on the group $\left(X,\,\circ \right)$. 
\begin{proof}
Observe that \eqref{affine2} and \eqref{affine1} trivially follow from the definition of $p$-brace. Moreover, by \cref{p-brace} (4), $\eta_{0}^{z_{i0}}=\id_X$. It also follows that $\eta^{z_{ij}}_{a}\eta^{z_{i\bar{j}}}_{a^{-1}} =  \eta^{z_{i\bar{j}}}_{a^{-1}}\eta^{z_{ij}}_{a} = \eta^{z_{i0}}_{0} = \id_X$, hence $\eta^{z_{ij}}_{a}\in Sym\left(X\right)$,  for all $a\in X$ and $z_{i,j} \in Y$, which completes the proof.
\end{proof}
\end{pro}

The following proposition shows that structures derived from \cref{p1}(2) are $p$-braces.
\begin{pro}\label{pro_pbrace}
    Let $(X, +, \circ)$ be a brace, $Y \subseteq {\cal D}(X)$, and suppose, for all $z_{i,j}\in Y,$  $z_i\circ z_j= z_j \circ z_i$. Let $\sigma_{a}^{z_{ij}}(b) := z_i^{-1} - a \circ z_i^{-1} \circ z_j + a \circ b\circ z_j$ 
    for all $a, b \in X$ and $z_{i,j} \in Y$, and set
    \begin{align*}
        a+_{ij}b:= a \circ\left( \sigma_a^{{z_{ij}}}\right)^{-1}(b).
    \end{align*}
    Then, the structure $\left(X,\,+_{ij}, \circ\right)$ is a $p$-brace.
    \begin{proof}
        Observing that $(\sigma^{z_{ij}}_a)^{-1}=\sigma^{z_{i\bar j}}_{a^{-1}}$, we get $a \circ (\sigma^{z_{ij}}_a)^{-1}(b) = b \circ (\sigma^{z_{ji}}_b)^{-1}(a)$, for all $a, b \in X$ and $z_{i,j} \in Y$. Moreover, $\sigma_{a}^{z_{ij}}\sigma_{b}^{z_{ik}} = \sigma_{a\circ b}^{z_{i,kj}}$, for all $a, b \in X$ and $z_{i, j,k }\in Y$.
Then $\left(\sigma^{z_{ij}}_a\right)^{-1}$ is a reversible $p$-affine structure on $(X, \circ)$ and the claim follows by \cref{thm_affine}.
\end{proof}
\end{pro}

\smallskip

\begin{exa}
Following \cite[Example 6]{MaRySt24}, let $A_8$ be the brace with additive group $(\mathbb{Z}/8\mathbb{Z},+)$ and multiplication given by $a\circ b=a+(-1)^ab,$ for all $a,b\in \mathbb{Z}/8\mathbb{Z}.$ Then $\mathcal{D}\left({A_8}\right)=\{0,2,4,6\}$. If one consider $Y=\{0, 2\}$, it is easy to check that the map $\eta^{z_{ij}}: \mathbb{Z}/8\mathbb{Z} \to \Sym\left(\mathbb{Z}/8\mathbb{Z}\right)$ given by
\begin{align*}
    \eta_a^{0,0}(b) = \eta_a^{2,2}(b) = \sigma_a(b)=(-1)^ab,&\qquad\qquad
    \eta_a^{2,0}(b)=\begin{cases}
       \ b &\text{if $a$ is even}\\
       \ 7b+4 &\text{if $a$ is odd}
    \end{cases},
    \end{align*}
    and
\begin{align*}
    \eta_a^{0,2}(b)&=\begin{cases}
       \ b &\text{if $a,b$ are even}\\
       \ b+4 &\text{if $a$ is even, $b$ is odd}\\
       \ 7b+4 &\text{if $a,b$ are odd}\\
       \ 7b &\text{if $a$ is odd, $b$ is even}
    \end{cases},
\end{align*}
for all $a, b \in \mathbb{Z}/8\mathbb{Z}$, is a $p$-affine structure. Consequently, the structure $\left(\mathbb{Z}/8\mathbb{Z}, +_{ij}, \circ\right)$ is a $p$-brace.
\end{exa}


\smallskip

\subsection*{B. General case} In \cite{St23}, the notion of \emph{semi-affine structure} on a group is introduced to study skew braces and more general structures. A similar definition can also be presented for the parametric case.

\begin{defn} 
    Let  $\left(X,\,\circ\right)$  be a group, $Y\leq X$, and $\eta^{z_{ij}}:X\to Sym(X)$ a
map from $\left(X,\,\circ\right)$ to the symmetric group on the set $X$ such that, for all $a,b\in X$ and $z_{i,j,k \in Y}$, \eqref{affine1} holds, i.e., 
\begin{align}
	&\eta^{z_{i, jk}}_{a\circ b} = \eta^{z_{ij}}_b\eta^{z_{ik}}_a
	= \eta^{z_{i, kj}}_{a\circ b}, \nonumber\\
	&\eta^{z_{0j}}_a\left(0\right) = 0, \label{eta_0}\\
	&\eta_{a}^{z_{ij,k}} = \eta_{a}^{z_{ji,k}}. \label{eta_comm}
\end{align}
Then, $\eta^{z_{ij}}$ is said to be a \emph{$p$-affine structure} on  $\left(X,\circ\right)$ if, for all $a,b\in X$ and $z_{i,j} \in Y$,
\begin{align}\label{affine2-gc}
	\eta^{z_{kj,i}}_{a}\left(b\circ \eta_{b}^{z_{kj}}\left(c\right)\right)
	= \eta_{a}^{z_{ji}}\left(b\right)\circ
	\eta_{\eta^{z_{ji}}_{a}\left(b\right)}^{z_{ki}}\eta_{a}^{z_{kj}}\left(c\right).
\end{align}
\end{defn}

\begin{rem}
     Any reversible $p$-affine structure $\eta^{z_{ij}}$ on $(X, \circ)$ is a $p$-affine structure. Indeed, \eqref{eta_0} follows from \cref{rem_eta} and \eqref{affine2} with $a=0$ and $i=0$. Moreover, if $a,b,c \in X$ and $z_{i, j, k} \in Y$, by \eqref{affine2} and \eqref{affine1}, equality in \eqref{eta_comm} follows from
    \begin{center}
    	$b\circ\eta^{z_{jk,i}}_{b}\left(a\right) 
    	= a\circ\eta^{z_{i,jk}}_{a}\left(b\right) 
    	= a\circ\eta^{z_{i,kj}}_{a}\left(b\right)
    	= b\circ\eta^{z_{kj,i}}_{b}\left(a\right)$.
    \end{center}
    In addition, 
    \begin{align*}
    	a\circ \eta_{a}^{z_{ji}}\left(b\right)\circ\eta_{\eta^{z_{ji}}_{a}\left(b\right)}^{z_{ki}}\eta_{a}^{z_{kj}}\left(c\right)
    	&= a\circ \eta_{a}^{z_{ji}}\left(b\right)\circ\eta^{z_{k,ij}}_{a\circ\eta^{z_{ji}}_{a}\left(b\right)}\left(c\right)\\
    	&= b\circ \eta_{b}^{z_{ij}}\left(a\right)\circ
    	\eta^{z_{k,ij}}_{b\circ\eta^{z_{ij}}_{b}\left(a\right)}\left(c\right)
    	&\mbox{by \eqref{affine2}}\\
    	&= b\circ \eta_{b}^{z_{ij}}\left(a\right)\circ
    	\eta^{z_{ki}}_{\eta^{z_{ij}}_{b}\left(a\right)}\eta_{b}^{z_{kj}}\left(c\right)\\
    	&= b\circ \eta_{b}^{z_{kj}}\left(c\right)\circ\eta^{z_{ik}}_{\eta^{z_{kj}}_b\left(c\right)}\eta_{b}^{z_{ij}}\left(a\right)&\mbox{by \eqref{affine2}}\\
    &=b\circ\eta_{b}^{z_{kj}}\left(c\right)\circ\eta^{z_{i,kj}}_{b\circ\eta_{b}^{z_{kj}}\left(c\right)}\left(a\right)\\
    	&= a\circ\eta^{z_{kj,i}}_{a}\left(b\circ \eta_{b}^{z_{kj}}\left(c\right)\right)&\mbox{by \eqref{affine2}}
    \end{align*}
    hence \eqref{affine2-gc} is satisfied.
\end{rem}

\begin{defn}\label{skew p-brace}
    Let $\left(X,\,\circ\right)$ be a group, $Y\leq X$, and $+_{i,j}$ a binary operation on $X$, for all $z_{i, j} \in Y$.
    A triple $\left(X, +_{ij}, \circ\right)$ is a \emph{skew $p$-brace} if, for all $a,b,c \in X$ and $z_{i, j, k} \in Y$
    \begin{enumerate}
	\item[$(1)$] $\left(a +_{ij} b\right) +_{k,ij} c = a +_{ki,j} \left(b +_{ki}c\right)$,
	\item[$(2)$] $a +_{i,jk} b= a +_{i,kj} b$,
    \item[$(3)$] $a +_{0i} 0 =a = 0 +_{i0} a$,
    \item[$(4)$] $\exists \vert x, y \in X$ such that $a +_{ij}x=0$ and $y +_{ij}a=0$,
    \end{enumerate}
    and \eqref{prop-p-brace} holds, i.e., $a\circ\left(b+_{k,ij}c\right) = a\circ b +_{ki} a\circ \left(a^{-1} +_{kj} c\right)$.  
\end{defn}

\begin{pro}
Let $\left(X, +_{ij}, \circ\right)$ be a skew $p$-brace and $Y \leq X$. Then, for all $z_{i,j} \in Y$, the additive structure $\left(X, +_{ij} \right)$ is a quasigroup. In particular,  $\left(X, +_{i0} \right)$ is a quasigroup with a left identity and $\left(X, +_{0i} \right)$ is a quasigroup with a right identity, for all $z_{i}\in Y$.
\begin{proof}
Initially, let us show that $\left(X, +_{ij}, \circ\right)$ is left and right cancellative. The left cancellativity can be proven as in \cref{rem_left_canc}. For the right one, let $a, b, x\in X$ and assume $a +_{ij} x = b +_{ij} x$. Then, considering $x'\in X$ such that $x+_{\bar{i}\, i} x' = 0$, by $(1)$ and $(3)$, we obtain that 
\begin{align*}
a&= a +_{0,j} 0 =
a+_{\bar{i}i, j} (x+_{\bar{i}\, i} x')
= (a+_{ij} x) +_{\bar{i}, ij} x' =  
(b+_{ij} x) +_{\bar{i}, ij} x'\\
&= b +_{\bar{i}\, i, j}(x+_{\bar{i},i}x')
= b +_{0,j} 0 = b.    
\end{align*}
The remaining of the proof can be proven as in \cref{rem_left_canc}.
\end{proof}
\end{pro}

Evidently, any $p$-brace is a skew $p$-brace. 
In the following, we will show that, as semi-affine structures on groups give rise to skew braces (see \cite{St23}), also $p$-affine structures provide skew $p$-braces and solutions.
\begin{pro}\label{prop_skewpbrace}
    Let $(X,\,\circ)$, be a group, $Y \leq X$, and $\eta^{z_{ij}}:X\to Sym(X)$ a $p$-affine structure. Define, for all $a,b\in X$ and $z_{i,j} \in Y$, the following binary operation on $X$:
\begin{align*}
a+_{ij}b:= a\circ\eta^{z_{ij}}_{a}\left(b\right). 
\end{align*}
Then, the triple $\left(X, +_{ij}, \circ \right)$ is a skew $p$-brace. 
\end{pro}
\begin{proof}
Let $a,b,c \in X$ and $z_{i,j,k} \in Y$. Then, 
\begin{align*}
    a +_{ki,j} \left(b+_{ki}c\right)
    &= a \circ \eta_a^{z_{ki,j}}\left(b \circ \eta_b^{z_{ki}}(c)\right)\\
    &=a\circ\eta_{a}^{z_{ij}}\left(b\right)\circ
    \eta_{\eta^{z_{ij}}_{a}\left(b\right)}^{z_{kj}}\eta_{a}^{z_{ki}}\left(c\right)&\mbox{by \eqref{affine2-gc}}\\
    &= a\circ\eta_{a}^{z_{ij}}\left(b\right)\circ\eta^{z_{k, ji}}_{a\circ\eta^{z_{ij}}_{a}\left(b\right)}\left(c\right)&\mbox{by \eqref{affine1}}\\    &=a\circ\eta_{a}^{z_{ij}}\left(b\right)\circ\eta^{z_{k, ij}}_{a\circ\eta^{z_{ij}}_{a}\left(b\right)}\left(c\right)&\mbox{by \eqref{affine1}}\\
    &= \left(a +_{ij} b\right) +_{k,ij} c,
\end{align*}
thus the equality $(1)$ of \cref{skew p-brace} is satisfied. Moreover, the identities $(2)$ and $a=0+_{i0} a$ in $(3)$ of \eqref{prop-p-brace} can be proved as in \cref{thm_affine}. In addition, by \eqref{eta_0}, we have $a +_{0i} 0 = a\circ\eta^{z_{0,i}}_{a}\left(0\right) = a$.
Finally, $a+_{ij} \eta_{a^{-1}}^{z_{i\bar j}}\left(a^{-1}\right)=0$ and $\eta^{z_{\bar i j}}_{a^{-1}}\left(a^{-1} \right) +_{{ij}} a=0$. Therefore, $\left(X, +_{ij}, \circ\right)$ is a skew $p$-brace.
  \end{proof}  

\begin{rem}
 Observe that in the case of a $p$-brace, we have $\eta^{z_{j\bar i}}_{a^{-1}}\left(a^{-1}\right) = \eta^{z_{\bar i j}}_{a^{-1}}\left(a^{-1}\right)$ and, actually, the left opposite of $a$ is consistent with that found in any skew $p$-brace.
\end{rem}

\begin{lemma}\label{lemma_prop_eta}
      Let $(X,\,\circ)$, be a group, $Y \leq X$, and $\eta^{z_{ij}}:X\to Sym(X)$ a $p$-affine structure. Define, for all $a,b\in X$ and $z_{i,j} \in Y$, the following binary operation on $X$ $a+_{ij}b:= a\circ\eta^{z_{ij}}_{a}\left(b\right)$. Then, the following hold for all $a, b, c \in X$ and $z_{i,j,k} \in Y$,
\begin{enumerate}
    \item[$(1)$] $\eta^{z_{kj,i}}_{a}\left(b +_{kj} c\right)
	= \eta_{a}^{z_{ji}}\left(b\right) +_{ki}\,
	\eta_{a}^{z_{kj}}\left(c\right)$, 
    \item[$(2)$] $\left(\eta^{z_{ij}}_a\right)^{-1}=\eta_{a^{-1}}^{z_{i\bar j}}$,
     \item[$(3)$] $\eta^{z_{\bar{i}\bar{i}}}_{a^{-1}}\left(a^{-1}\right) +_{i\bar{i}} a=0$,
    \item[$(4)$] $\eta^{z_{ji}}_{a^{-1}}\left(a^{-1}\right) +_{ki} a\circ b=a\circ\left(a^{-1} +_{kj,i} b\right)$.  
\end{enumerate}
\begin{proof}
Initially, the equality in $(1)$ follows directly by \eqref{affine2-gc}. Moreover, let $a, b \in X$ and $z_{i,j,k} \in Y$. From the proof of \cref{prop_skewpbrace},
   $
   \eta_0^{z_{i0}}=\id_X$ and $\eta^{z_{\bar{i}\, \bar{i}}}_{a^{-1}}\left(a^{-1}\right) +_{i\bar{i}} a
	=0$.
Furthermore, by $(1)$ and $(2)$, we have
\begin{align*}
	\eta^{z_{ji}}_{a^{-1}}\left(a^{-1}\right) +_{ki} a\circ b
	&= \eta^{z_{ji}}_{a^{-1}}\left(a^{-1}\right) +_{ki} \eta^{z_{kj}}_{a^{-1}}\eta^{z_{k\bar{j}}}_{a}\left(a\circ b\right)
	= \eta^{z_{kj,i}}_{a^{-1}}\left(a^{-1} +_{{kj}} \eta^{z_{k\bar{j}}}_{a}\left(a\circ b\right)\right)\\
	&= \eta^{z_{kj,i}}_{a^{-1}}\left(b\right)
	= a\circ\left(a^{-1} +_{kj,i} b\right).
    \qedhere
    \end{align*}
\end{proof}
\end{lemma}

\begin{thm}\label{thm_skew_p}
    Let $(X,\,\circ)$, be a group, $Y \leq X$, and $\eta^{z_{ij}}:X\to Sym(X)$ a $p$-affine structure. Define, for all $a,b\in X$ and $z_{i,j} \in Y$, the following binary operation on $X$ $a+_{ij}b:= a\circ\eta^{z_{ij}}_{a}\left(b\right)$. Then, the maps defined by 
\begin{center}
    $\sigma^{z_{ij}}_a(b):= (\eta_{a}^{z_{ij}})^{-1}(b)=a \circ \left(a^{-1}+_{i\bar{j}} b\right)$ \quad \text{and} \quad $\tau^{z_{ij}}_b\left(a\right):=
\sigma_a^{z_{ij}}\left(b\right)^{-1}\circ a\circ b$, 
\end{center}for all $a, b \in X$ and $z_{i,j}\in Y$, determine a solution $R^{z_{ij}}(b, a)=\left(\sigma^{z_{ij}}_a(b), \tau^{z_{ij}}_b\left(a\right)\right)$ on $X$.
\begin{proof} 
First, let us note that by $(2)$ in \cref{lemma_prop_eta}, $\sigma_{0}^{z_{i0}} = \id_X$ and $(\sigma_{a}^{z_{ij}})^{-1} = \sigma_{a^{-1}}^{z_{i\bar{j}}}$, for all $a\in X$ and $z_{i,j} \in Y$.  To get the claim, we use  \cref{le:lndsol} and so we show that, for all $a, b \in X$ and $z_{i,j} \in Y$, the map $\sigma^{z_{ij}}_a$ is an admissible twist for the structure $\left(X , \triangleright_{z_{ij}}\right)$ that we are going to compute and prove is a $p$-rack. If $a, b \in X$ and $z_{i,j} \in Y$, we have
\begin{align*}
        a\triangleright_{z_{ij}} b
     =\sigma^{z_{ji}}_{a}\tau^{z_{ij}}_{\left(\sigma^{z_{ij}}_b\right)^{-1}\left(a\right)}\left(b\right) = a\circ\left(a^{-1} +_{j\bar{i}} a^{-1}\circ\left(b+_{ij}a\right)\right)
     = \eta^{z_{j\bar{i}}}_{a^{-1}}\left(a^{-1}\circ b\circ\eta^{z_{ij}}_{b}\left(a\right)\right).
\end{align*}
Let us prove that $\left(X, \triangleright_{z_{ij}}\right)$ is a $p$-rack. If $a, b, c \in X$ and $z_{i, j, k} \in Y$, then
\begin{align*}
\left(a\triangleright_{z_{ij}} b\right)\triangleright_{z_{jk}}\left(a\triangleright_{z_{ik}} c\right)
    = \eta^{z_{k\bar{j}}}_{\left(a \, \triangleright_{z_{ij}} b\right)^{-1}}\left(\underbrace{\left(a\triangleright_{z_{ij}} b\right)^{-1}\circ\left(a\triangleright_{z_{ik}} c\right)\circ
    \eta^{z_{jk}}_{a \, \triangleright_{z_{ik}} c}
    \left(a\triangleright_{z_{ij}} b\right)}_{A}
    \right).
\end{align*}
Note that
\begin{align*}
    &\left(a\triangleright_{z_{ik}} c\right)\circ
    \eta^{jk}_{a \, \triangleright_{z_{ik}} c}
    \left(a\triangleright_{z_{ij}} b\right)\\
    &\ = \eta^{z_{k\bar{i}}}_{a^{-1}}\left(a^{-1}\circ c\circ\eta^{z_{ik}}_{c}\left(a\right)\right)\circ
    \eta^{z_{j\bar{i}}}_{
    \eta^{z_{k\bar{i}}}_{a^{-1}}\left(a^{-1}\circ \, c \, \circ \, \eta^{z_{ik}}_{c}\left(a\right)\right)}
    \eta^{z_{jk}}_{a^{-1}}\left(a^{-1}\circ b\circ\eta^{z_{ij}}_{b}\left(a\right)\right)&\mbox{by \eqref{affine1}}\\
    &\ = \eta^{z_{jk,\bar{i}}}_{a^{-1}}
    \left(a^{-1}\circ c\circ\eta^{z_{ik}}_{c}\left(a\right)\circ\eta^{z_{jk}}_{a^{-1} \, \circ \,c\circ \,\eta^{z_{ik}}_{c}\left(a\right)}\left(a^{-1}\circ b\circ\eta^{z_{ij}}_{b}\left(a\right)\right)\right)&\mbox{by \eqref{affine2-gc}}\\
    &\ = \eta^{z_{jk,\bar{i}}}_{a^{-1}}
    \left(a^{-1}\circ c\circ\eta^{z_{ik}}_{c}\left(a\right)\circ\eta^{z_{jk}}_{\eta^{z_{ik}}_{c}\left(a\right)}
    \eta^{z_{ji}}_{c}\eta^{z_{j\bar{i}}}_{a^{-1}}\left(a^{-1}\circ b\circ\eta^{z_{ij}}_{b}\left(a\right)\right)\right)&\mbox{by \eqref{affine1}}\\
    &\ =\eta^{z_{jk,\bar{i}}}_{a^{-1}}
    \left(a^{-1}\circ c\circ
    \eta^{z_{ji,k}}_{c}\left(a\circ \eta^{z_{ji}}_{a}\eta^{z_{j\bar{i}}}_{a^{-1}}\left(a^{-1}\circ b\circ\eta^{z_{ij}}_{b}\left(a\right) \right)\right)\right)&\mbox{by \eqref{affine2-gc}}\\
    &\ =\eta^{z_{kj,\bar{i}}}_{a^{-1}}
    \left(a^{-1}\circ c\circ
    \eta^{z_{ij,k}}_{c}\left(b\circ \eta^{z_{ij}}_{b}\left(a\right) \right)\right),&\mbox{by \eqref{affine1}}
\end{align*}
hence, by \eqref{affine2-gc},
\begin{align*}
    A&= \big(\eta^{z_{j\bar{i}}}_{a^{-1}}\left(a^{-1}\circ b\circ \eta^{z_{ij}}_{b}\left(a\right)\right)\big)^{-1}\circ
    \eta^{z_{kj,\bar{i}}}_{a^{-1}}
    \left(a^{-1}\circ c\circ
    \eta^{z_{ij,k}}_{c}\left(b\circ \eta^{z_{ij}}_{b}\left(a\right) \right)\right)\\
    &= \eta^{z_{k\bar{i}}}_{\eta^{z_{j\bar{i}}}_{a^{-1}}\left(a^{-1}\circ\, b\,\circ \,\eta^{z_{ij}}_{b}\left(a\right)\right)}
    \eta^{z_{k0}}_{\left(a^{-1}\, \circ \, b\, \circ \,\eta^{z_{ij}}_{b}\left(a\right)\right)^{-1}\circ a^{-1}}\left(\left(\eta^{z_{ij}}_{b}\left(a\right)\right)^{-1}\circ b^{-1}\circ c\circ\eta^{z_{ij,k}}_{c}\left(c\circ\eta^{z_{ij}}_{b}\left(a\right)\right)\right)
    \\
    &=\eta^{z_{k\bar{i}}}_{a\, \triangleright_{z_{ij}}b}\,
    \eta^{z_{k0}}_{\left(\eta^{z_{ij}}_{b}\left(a\right)\right)^{-1}\circ \, b^{-1}}\left(\left(\eta^{z_{ij}}_{b}\left(a\right)\right)^{-1}\circ b^{-1}\circ c\circ\eta^{z_{ij,k}}_{c}\left(c\circ\eta^{z_{ij}}_{b}\left(a\right)\right)\right).
\end{align*}
It follows that by \eqref{affine1}
\begin{align*}
    \left(a\triangleright_{z_{ij}} b\right)\triangleright_{z_{jk}}\left(a\triangleright_{z_{ik}} c\right)
    = \eta^{z_{k,\bar{j}\,\bar{i}}}_{\left(\eta^{z_{ij}}_{b}\left(a\right)\right)^{-1}\circ b^{-1}}\left(\left(\eta^{z_{ij}}_{b}\left(a\right)\right)^{-1}\circ b^{-1}\circ c\circ\eta^{z_{ij,k}}_{c}\left(c\circ\eta^{z_{ij}}_{b}\left(a\right)\right)\right).
\end{align*}
Now, observe that
\begin{align*}
    a\triangleright_{z_{ik}}\left(b\triangleright_{z_{jk}} c\right)
        = \eta^{z_{k\bar{i}}}_{a^{-1}}
        \big(\underbrace{a^{-1}\circ\eta^{z_{k\bar{j}}}_{b^{-1}}
        \left(
        b^{-1}\circ c\circ \eta^{z_{jk}}_{c}\left(b\right)
        \right)
        \circ\eta^{z_{ik}}_{\eta^{z_{k\bar{j}}}_{b^{-1}}
        \left(
        b^{-1}\, \circ \, c\,\circ \,\eta^{z_{jk}}_{c}\left(b\right)\right)}\left(a\right)}_{B}
        \big)
\end{align*}
Since
\begin{align*}
    \eta^{z_{k\bar{j}}}_{b^{-1}}
        &\left(
        b^{-1}\circ c\circ \,\eta^{z_{jk}}_{c}\left(b\right)
        \right)
        \circ\eta^{z_{ik}}_{\eta^{z_{k\bar{j}}}_{b^{-1}}
        \left(
        b^{-1}\, \circ\, c\, \circ \,\eta^{z_{jk}}_{c}\left(b\right)\right)}\left(a\right)\\
        &\ =
        \eta^{z_{k\bar{j}}}_{b^{-1}}
        \left(
        b^{-1}\,\circ c \, \circ \,\eta^{z_{jk}}_{c}\left(b\right)
        \right)
        \circ\eta^{z_{i\bar{j}}}_{\eta^{z_{k\bar{j}}}_{b^{-1}}
        \left(
        b^{-1}\, \circ \, c\, \circ \, \eta^{z_{jk}}_{c}\left(b\right)\right)}
        \eta^{z_{ik}}_{b^{-1}}\eta^{z_{ij}}_{b}\left(a\right)\\
        &\ = \eta^{z_{ik,\bar{j}}}_{b^{-1}}
        \big(b^{-1}\circ c\circ \eta^{z_{jk}}_{c}\left(b\right)\circ\eta^{z_{ik}}_{b^{-1}\circ \, c\, \circ\, \eta^{z_{jk}}_{c}\left(b\right)}\eta^{z_{ij}}_{b}\left(a\right)\big)&\mbox{by \eqref{affine2-gc}}\\
        &\ = \eta^{z_{ik,\bar{j}}}_{b^{-1}}
        \big(b^{-1}\circ c\circ \eta^{z_{jk}}_{c}\left(b\right)\circ\eta^{z_{ik}}_{\eta^{z_{jk}}_{c}\left(b\right)}\eta^{z_{ij}}_{c}\left(a\right)\big)&\mbox{by \eqref{affine1}}\\
        &\ = \eta^{z_{ik,\bar{j}}}_{b^{-1}}
        \big(b^{-1}\circ c\circ \eta^{z_{ij,k}}_{c}\left(b\circ\eta^{z_{ij}}_{b}\left(a\right)\right)\big),
        &\mbox{by \eqref{affine2-gc}}
\end{align*}
by \eqref{affine2-gc}, we obtain that
\begin{align*}
   B&=\big(\eta^{z_{i\bar{j}}}_{b^{-1}}\left(\eta^{z_{ij}}_{b}\left(a\right)\right)\big)^{-1}
    \circ\eta^{z_{k\bar{j}}}_{b^{-1}}
        \left(\eta^{z_{ij}}_{b^{-1}}\left(a\right)\circ\left(\eta^{z_{ij}}_{b^{-1}}\left(a\right)\right)^{-1}\circ
        b^{-1}\circ c\circ \eta^{z_{jk}}_{c}\left(b\right)
        \right) \circ \\
&\qquad\circ\eta^{z_{ik}}_{\eta^{z_{k\bar{j}}}_{b^{-1}}
        \left(
        b^{-1} \, \circ \, c\, \circ \,\eta^{z_{jk}}_{c}\left(b\right)\right)}\left(a\right) \\
    &\ = \eta^{z_{k\bar{j}}}_{\eta^{z_{i\bar{j}}}_{b^{-1}} \eta^{z_{ij}}_{b}\left(a\right)}
    \eta^{z_{k,0}}_{\left(\eta^{z_{ij}}_{b}\left(a\right)\right)^{-1}\circ \,  b^{-1}}
    \left(\left(\eta^{z_{ij}}_{b^{-1}}\left(a\right)\right)^{-1}\circ b^{-1}\circ c\circ \eta^{z_{ij,k}}_{c}\left(b\circ\eta^{z_{ij}}_{b}\left(a\right)\right)\right)\\
    &\ = \eta^{z_{k\bar{j}}}_{a}
    \eta^{z_{k,0}}_{\left(\eta^{z_{ij}}_{b}\left(a\right)\right)^{-1}\circ \,b^{-1}}\left(\left(\eta^{z_{ij}}_{b^{-1}}\left(a\right)\right)^{-1}\circ b^{-1}\circ c\circ \eta^{z_{ij,k}}_{c}\left(b\circ\eta^{z_{ij}}_{b}\left(a\right)\right)\right).
\end{align*}
Consequently, by \eqref{affine1},
\begin{align*}
a\triangleright_{z_{ik}}\left(b\triangleright_{z_{jk}} c\right)
    &= 
    \eta^{z_{k, \bar{i}\,\bar{j}}}_{\left(\eta^{z_{ij}}_{b}\left(a\right)\right)^{-1}\circ b^{-1}}\left(\left(\eta^{z_{ij}}_{b^{-1}}\left(a\right)\right)^{-1}\circ b^{-1}\circ c\circ \eta^{z_{ij,k}}_{c}\left(b\circ\eta^{z_{ij}}_{b}\left(a\right)\right)\right).
\end{align*}
Observing that,  by \eqref{affine1}, $\eta^{z_{k, \bar{i}\,\bar{j}}}=\eta^{z_{k, \bar{j}\,\bar{i}}}$, condition \eqref{shelf} is satisfied and so $\left(X, \triangleright_{z_{ij}}\right)$ is a $p$-shelf.\\
Now, we show that the maps $L_a^{z_{ij}}: X \to X$ are bijective.
%
%
%
%
Assume that $L_{a}^{z_{ij}}\left(b\right) = L_{a}^{z_{ij}}\left(c\right)$. Then 
$b +_{ij} a = c +_{ij} a$, hence, by \cref{prop_skewpbrace}, for the right cancellativity of $+_{ij}$, we have $b=c$.
Now, let $y\in X$ and set $x:= a\circ\eta^{z_{ji}}_a\left(y\right) +_{\bar{i}, ji}\eta^{z_{\bar{i}\,\bar{i}}}_{a^{-1}}\left(a^{-1}\right)$.  Then
\begin{align*}
    x +_{ij} a 
    &= \left(a\circ\eta^{z_{ji}}_a\left(y\right) +_{\bar{i}, ji}\eta^{z_{\bar{i}\,\bar{i}}}_{a^{-1}}\left(a^{-1}\right)\right) +_{ij} a \\
    &= a\circ\eta^{z_{ji}}_a\left(y\right) +_{i\,\bar{i}, ji} \left(\eta^{z_{\bar{i}\, \bar{i}}}_{a^{-1}}\left(a^{-1}\right) +_{i\bar{i}} a\right)&\mbox{ $\left(X, +_{ij}, \circ \right)$ is a skew $p$-brace}\\
    &= a\circ\eta^{z_{ji}}_a\left(y\right) +_{0, ji} 0 &\mbox{by $(3)$ of \cref{lemma_prop_eta}}\\
    &= a\circ\eta^{z_{ji}}_a\left(y\right) &\mbox{ $\left(X, +_{ij}, \circ \right)$ is a skew $p$-brace}
\end{align*}
Hence, $
    L^{z_{ij}}_a\left(x\right)
     = \eta^{z_{j\bar{i}}}_{a^{-1}}\left(a^{-1}\circ \left(x +_{ij} a\right)\right)
     = \eta^{z_{j\bar{i}}}_{a^{-1}}\eta^{z_{ji}}_{a}\left(y\right) = y$, i.e., $L^{z_{ij}}_a$ is surjective.\\
     Finally, since $\sigma^{z_{ij}}_a(b) \circ \tau^{z_{ij}}_b(a)=a\circ b$
and, by \eqref{affine1} $\sigma^{z_{i, jk}}_{a\circ b} = \sigma^{z_{i, kj}}_{a\circ b}$, for all $a,b \in X$ and $z_{i, j, k} \in Y$, we obtain condition $(1)$ in \cref{admi2}.  Moreover, 
\begin{align*}
    \sigma^{z_{ik}}_c\left(b\right)&\triangleright_{z_{ij}}\sigma^{z_{jk}}_c\left(a\right)
    = \eta^{z_{j\bar{i}}}_{\left(\eta^{z_{i\bar{k}}}_{c^{-1}}\left(b\right)\right)^{-1}}
    \big(\left(\eta^{z_{i\bar{k}}}_{c^{-1}}\left(b\right)\right)^{-1}\circ \eta^{z_{j\bar{k}}}_{c^{-1}}\left(a\right)\circ
    \eta^{z_{ij}}_{\eta^{z_{j\bar{k}}}_{c^{-1}}\left(a\right)}\eta^{z_{i\bar{k}}}_{c^{-1}}\left(b\right)\big)\\
    &=\eta^{z_{j\bar{i}}}_{\left(\eta^{z_{i\bar{k}}}_{c^{-1}}\left(b\right)\right)^{-1}}
    \big(\left(\eta^{z_{i\bar{k}}}_{c^{-1}}\left(b\right)\right)^{-1}\circ \eta^{z_{j\bar{k}}}_{c^{-1}}\left(a\right)\circ
    \eta^{z_{i\bar{k}}}_{\eta^{z_{j\bar{k}}}_{c^{-1}}\left(a\right)}\eta^{z_{ij}}_{c^{-1}}\left(b\right)\big)&\mbox{by \eqref{affine1}}\\
    &=\eta^{z_{j\bar{i}}}_{\left(\eta^{z_{i\bar{k}}}_{c^{-1}}\left(b\right)\right)^{-1}}
    \big(\left(\eta^{z_{i\bar{k}}}_{c^{-1}}\left(b\right)\right)^{-1}\circ \eta^{z_{ji,\bar{k}}}_{c^{-1}}\left(a\circ\eta^{z_{ij}}_{a}\left(b\right)\right)\big)&\mbox{by \eqref{affine2-gc}-\eqref{eta_comm}}\\
    &= \eta^{z_{j\bar{i}}}_{\left(\eta^{z_{i\bar{k}}}_{c^{-1}}\left(b\right)\right)^{-1}}
    \big(\left(\eta^{z_{i\bar{k}}}_{c^{-1}}\left(b\right)\right)^{-1}\circ \eta^{z_{ji,\bar{k}}}_{c^{-1}}\left(b\circ\eta^{z_{ji}}_{b}\left(\eta^{z_{j\bar{i}}}_{b^{-1}}\left(b^{-1}\circ a\circ\eta^{z_{ij}}_{a}\left(b\right)\right)\right)\right)\\
    &= \eta^{z_{j\bar{i}}}_{\left(\eta^{z_{i\bar{k}}}_{c^{-1}}\left(b\right)\right)^{-1}}
    \big(\left(\eta^{z_{i\bar{k}}}_{c^{-1}}\left(b\right)\right)^{-1}\circ \eta^{z_{i\bar{k}}}_{c^{-1}}\left(b\right)\circ \eta^{z_{j\bar{k}}}_{\eta^{z_{i\bar{k}}}_{c^{-1}}\left(b\right)}\eta^{z_{ji}}_{c^{-1}}\eta^{z_{j\bar{i}}}_{b^{-1}}\left(b^{-1}\circ a\circ\eta^{z_{ij}}_{a}\left(b\right)\right)\big) &\mbox{by \eqref{affine2-gc}}\\
    &=\eta^{z_{j\bar{i}}}_{\left(\eta^{z_{i\bar{k}}}_{c^{-1}}\left(b\right)\right)^{-1}}\eta^{z_{j\bar{k}}}_{b^{-1}\circ c^{-1}\circ \eta^{z_{i\bar{k}}}_{c^{-1}}\left(b\right)}\left(b^{-1}\circ a\circ\eta^{z_{ij}}_{a}\left(b\right)\right) &\mbox{by \eqref{affine1}}\\
    &= \eta^{z_{j,\bar{k}\bar{i}}}_{b^{-1}\circ c^{-1}}\left(b^{-1}\circ a\circ\eta^{z_{ij}}_{a}\left(b\right)\right)&\mbox{by \eqref{affine1}}\\   &=\sigma^{z_{jk}}_c\left(b\triangleright_{z_{ij}}a\right), &\mbox{by \eqref{affine1}}
\end{align*}
i.e., \cref{admi2} (2) is satisfied. Consequently maps $\sigma_a^{z_{ij}}$ determine an admissible twist and by \cref{le:lndsol} the claim follows.
\end{proof}
\end{thm}

\begin{rem}
Let us observe that the shelf operation $\triangleright_{z_{ij}}$ can be written as 
\begin{align}\label{triangle}
    a \triangleright_{z_{ij}} b = \eta^{z_{j \bar i}}_{a^{-1}}\left(a^{-1}\right) +_{0\bar{i}} \left(b+_{ij} a\right)
\end{align}
for all $a, b \in X$ and $z_{i,j}\in Y$, that is indeed the conjugation quandle $a \triangleright b=-a+b+a$ in the non-parametric case. Indeed, by $(1)$ in \cref{lemma_prop_eta},
\begin{align*}
    a \triangleright_{z_{ij}} b=\eta^{z_{j\bar{i}}}_{a^{-1}}\left(a^{-1}\circ b\circ\eta^{z_{ij}}_{b}\left(a\right)\right)
    =\eta^{z_{j\bar{i}}}_{a^{-1}}\left(a^{-1} \right)+_{0\bar{i}} b \circ \eta_b^{z_{ij}}(a)=\eta^{z_{j \bar i}}_{a^{-1}}\left(a^{-1}\right) +_{0\bar{i}} \left(b+_{ij} a\right),
\end{align*} 
for all $a,b\in X$ and $z_{i, j}\in Y$.
\end{rem}

\begin{pro}
Let $\left(X, +_{ij}, \circ \right)$ be a skew $p$-brace and $Y \leq X$. Then, the map $\eta^{z_{ij}}_a:X\to X$, given by $\eta^{z_{ij}}_{a}\left(b\right):= a^{-1}\circ\left(a +_{ij} b\right)$, 
for all $a,b\in X$ and $z_{i,j} \in Y$, defines a $p$-affine structure $\eta^{z_{ij}}$ on the group $\left(X,\,\circ \right)$.
\begin{proof}
The proof is similar to that of \cref{pro_p_brace}.
\end{proof}
\end{pro}

The following result shows that the structures derived from \cref{p2} (2) are skew $p$-braces.
\begin{pro}
    Let $(X, +, \circ)$ be a skew brace, $Y \subseteq {\cal D}(X) \cap Z(X, +)$, and suppose, for all $z_{i,j}\in Y,$  $z_i\circ z_j= z_j \circ z_i$. Let $\sigma_{a}^{z_{ij}}(b) := z_i^{-1} - a \circ z_i^{-1} \circ z_j + a \circ b\circ z_j$ 
    for all $a, b \in X$ and $z_{i,j} \in Y$, and set
    \begin{align*}
        a+_{ij}b:= a \circ\left( \sigma_a^{{z_{ij}}}\right)^{-1}(b).
    \end{align*}
    Then, the structure $\left(X,\,+_{ij}, \circ\right)$ is a skew $p$-brace.
    \begin{proof}
     The proof is similar to that of \cref{pro_pbrace}, since it is enough to show that the map $\left(\sigma_{a}^{z_{ij}}\right)^{-1}$ provides a $p$-affine structure.
    \end{proof}
\end{pro}

\begin{rem}
    We observe that the binary operation $\bullet_{z_{ij}}$ defined in \cref{p1} and \cref{p2} can be written in particular in terms of $+_{ij}$ in \cref{thm_affine} and \cref{prop_skewpbrace}, namely, as $a\bullet_{z_{ji}} b = \left(a+_{ij} b\right)\circ h\left(z_{ij}\right)$, for all $a,b\in X$ and $z_{i,j}\in Y$.
\end{rem}

\smallskip

\section{\texorpdfstring{$p$-}{}rack Yang-Baxter \texorpdfstring{$\&$}{} reflection operators}

\noindent In this section, we study the underlying Yang-Baxter and reflection algebras, and their coproduct structures in the set-theoretic frame.
Specifically, we focus on the reflection algebraic structures associated to $p$-rack solutions of Yang-Baxter equation.
Bearing in mind the definition of braided groups and braidings in \cite{chin}, \cite{GatMaj} 
and their deformations \cite{DoiRyb22, DoRy23, DoRySt} we further
generalize the definitions to introduce the parametric Yang-Baxter and reflection structures. The
parametric Yang-Baxter operator and associated algebraic structures were studied in \cite{Doikoup}.

It is useful to introduce the following notation. Let $X$, $Y \subseteq X$ be non-empty sets and introduce, for all 
$z_{i, j, k} \in Y$, the maps 
$M^{z_{ijk}}_b \in \{g^{z_{ijk}}_{b}, \hat g^{z_{ijk}}_{b} \},$ $M_b^{z_{ijk}}: X \to X,$
$a \mapsto 
M^{z_{ijk}}_b(a)$ and the map $\tau_b^{z_{ij}}: X \to X,$ $a \mapsto \tau^{z_{ij}}_b(a).$

\begin{defn} \label{defA}
Let $X$ and $Y \subseteq X$ be non-empty sets and let, for all $z_{i,j} \in Y$,
$\bullet_{z_{ij}}: X \times X \to X,$ $(a, b) \mapsto a \bullet_{z_{ij}} b$ be a binary operation. 
Consider also, for all $ z_{i,j} \in Y$, the following maps, $m_{z_{ij}}: X \times X \to X,\ (a,b)\mapsto a \bullet_{z_{ij}} b,$
$~\pi(a,b)=(b,a)$,
\[R^{z_{ij}}(b,a) = (b,  \tau^{z_{ij}}_b(a)),\quad \xi^{z_{ijk}}(b, a) = (b,  g_b^{z_{ijk}}(a)),  \quad  \zeta^{z_{ijk}}(b,a) = (b, \hat g^{z_{ijk}}_b(a)),\]
such that $\tau_b^{z_{ij}},\ g_b^{z_{ijk}},\ \hat g_b^{z_{ijk}}$ are bijections for all $a, b \in X,$ $z_{i,j,k} \in Y$ and
$g^{z_{ijk}}_b(a) = g^{z_{ikj}}_b(a),$ $\hat g^{z_{ijk}}_b(a) = \hat g^{z_{jik}}_b(a)$ and $K^{z_i}: X \to X,$ $a\mapsto \kappa^{z_i}(a),$ where $\kappa^{z_i}$ is a bijection for all $z_i \in Y.$ Let also $\hat m_{z_{ij}} := m_{z_{ij}} \pi.$ 
Then $(X, \bullet_{z_{ij}})$ is called a \emph{$p$-rack magma}, the map $R^{z_{ij}}$ is called  a 
\emph{$p$-rack operator} and the map $K^z$ is called a \emph{$p$-rack reflection operator}, if for all 
$a,b,c \in X,$ $z_{i,j,k} \in Y:$
\begin{enumerate}
\item[$(1)$] $\hat m_{z_{ji}} (b, a) =m_{z_{ij}} (R^{z_{ij}} (b, a)).$
\item[$(2)$] $\xi^{z_{ikj}} (\id_X \times \hat m_{kj} ) (c,b,a)= (\id_X \times \hat  m_{kj}) R_{13}^{z_{ik}} R_{12}^{z_{ij}} (c,b,a).$
\item[$(3)$] $\zeta^{z_{ijk}}(\hat m_{z_{ji}} \times \id_X )(c,b,a)= (\hat m_{z_{ji}} \times \id_X)  R_{13}^{z_{ik}}R_{23}^{z_{jk}}(c,b,a).$
\item[$(4)$] $(\id_X \times K^{z_j})R^{z_{i\bar j}}(a,b)= R^{z_{ij}} (\id_X \times K^{z_j})(a,b).$
\item[$(5)$] $(K^{z_i} \times \id_X) R^{z_{\bar i j}}(a,b)=R^{z_{i j}}(K^{z_i} \times \id_X)(a,b).$
\end{enumerate}
\end{defn}

\begin{pro}\label{le:p-def1}
Let $(X,\bullet_{z_{ij}})$ be a $p$-rack magma, consider a map 
$R^{z_{ij}}: X \times X \to X \times X,$ $R^{z_{ij}}(b,a) = (b,  \tau^{z_{ij}}_b(a))$, such that $R^{z_{ij}}$ is a 
$p$-rack operator, and $K^{z_{i}}: X \to X,$ $a \mapsto \kappa^{z_{i}}(a)$ a $p$-rack reflection operator, for all $z_{i,j}\in Y.$
Then,
\begin{enumerate}
\item[$(1)$] $R^{z_{ij}}$ satisfies the parametric Yang-Baxter equation.

\item[$(2)$] $K^{z_i}$ satisfies the reflection equation.

\item[$(3)$]  $\xi^{z_{ijk}}_{12}K_1^{z_i} \zeta^{z_{jk\bar i}}_{21} (\id_X \times \hat m_{z_{kj}})(c,b,a)=(\id_X \times \hat m_{z_{kj}})R^{z_{ik}}_{13} R^{z_{ij}}_{12}K_1^{z_i}R^{z_{j\bar i}}_{21}R^{z_{k\bar i}}_{31} (c,b,a).$
\end{enumerate}
\end{pro}
\begin{proof} 

$ $

\begin{enumerate}
\item The proof of (1) is based on the Definition \ref{defA} (see \cite{Doikoup} for the detailed proof). 

\item To show that $K^{z_i}$ satisfies the reflection equation we just obtain the conditions from (4) and (5) of Definition \ref{defA}: 
(i) $~a \triangleright_{z_{12}} \kappa^{z_2}(b) = \kappa^{z_2}(a \triangleright_{z_{1\bar 2}} b)$ 
and (ii) $~\kappa^{z_1}(a) \triangleright_{z_{ 12}} b =a \triangleright_{z_{\bar 1 2}} b,$ for all $a, b\in X,$ $z_{1,2} \in Y.$ 
These are precisely the conditions satisfied by any solution of the reflection equation with 
$R^{z_{ij}}$ being the $p$-rack operator, see \cref{rackf}.

\item The proof of part (3) is based on the action of the involved maps on $(c,b,a)$ and the 
use of the two main properties (2) and (3) of Definition \ref{defA}.
\end{enumerate}
\end{proof}

\begin{exa} \label{bullet}
    Let $(X, +, \circ)$ be a skew brace, $Y \subseteq {\cal D}(X) \cap Z(X,+)$, and recall the binary operations from Proposition \ref{p2}, 
    $\bullet_{z_{ij}},\ \triangleright_{z_{ij}}: X\times X \to X,$ such that 
    $a\bullet_{z_{ij}}b =a \circ z_i + b \circ z_j$ and 
    $a\triangleright_{z_{ij}} b  = - a \circ z_{i} \circ z_j^{-1} + b + a\circ z_i \circ z_j^{-1},$ 
    for all $a,b\in X,$ $z_{i,j} \in Y.$ Then the following relations are satisfied:
    \[a \bullet_{z_{ji}} b = b \bullet_{z_{ij
}} (b \triangleright_{z_{ij}} a), \quad \mbox{and}  \quad 
b \triangleright_{z_{jk}} (a\triangleright_{z_{ik}} c) 
 = (a\bullet_{z_{ji}}b) \triangleright_{z_{0k}} c,\]
where $z_0 =0$. Hence the function $\hat g^{z_{ijk}}_y(x)$ of 
Definition \ref{defA} is $\hat g^{z_{ijk}}_b(a)=: \hat g^{z_{0k}}_b(a) = b \triangleright_{z_{0k}} a.$ Thus, $\left(X, \bullet_{z_{ij}}\right)$ is a $p$-rack magma and the $p$-rack solution $R^{z_{ij}}$ is a $p$-rack operator. Moreover, the map $K^{z_i}: X \to X$  given by $\kappa^{z_i}(a) = a \circ z_i^{-1} \circ z_i^{-1} + m\zeta$, with $m \in \mathbb{Z}$ and $\zeta \in Z(X,+),$ provided in \cref{ex_conj} is a $p$-rack reflection operator.

\end{exa}

\bigskip

\section*{Acknowledgements}
\noindent A. Doikou acknowledges support from the EPSRC research grant EP/V008129/1. \\The Department of Mathematics and Physics 
“Ennio De Giorgi”,  University of Salento, partially supported this work. M. Mazzotta and P. Stefanelli are members of GNSAGA (INdAM) and the non-profit association ADV-AGTA.

\section*{Declarations}

\subsection*{Declarations}
Not applicable.

\subsection*{Ethical Approval}
Not applicable.

\subsection*{Funding}
A. Doikou is supported by the EPSRC research grant EP/V008129/1.

\subsection*{Availability of data and materials}
No new data were created or analysed in this study. Data sharing is not applicable to this article.

\end{document}